\documentclass[11pt,a4paper,leqno]{amsart}
\usepackage[latin1]{inputenc}
\usepackage[T1]{fontenc}
\usepackage{amsfonts}
\usepackage{amsmath}
\usepackage{amssymb}
\usepackage{eurosym}
\usepackage{mathrsfs}
\usepackage{palatino}
\usepackage{pxfonts}
\usepackage{color}
\usepackage{esint}
\usepackage{url}
\usepackage{verbatim}
\usepackage{mathtools}
\mathtoolsset{showonlyrefs}

\usepackage{enumerate}

\usepackage[pagebackref,hypertexnames=false, colorlinks, citecolor=blue, linkcolor=blue, urlcolor=red]{hyperref}

\newcommand{\C}{\mathbb{C}}
\newcommand{\Z}{\mathbb{Z}}

\newcommand{\N}{\mathbb{N}}
\newcommand{\bbP}{\mathbb{P}}

\newcommand{\R}{\mathbb{R}}

\newcommand{\E}{\mathbb{E}}

\newcommand{\calA}{\mathcal{A}}

\newcommand{\calJ}{\mathcal{J}}

\newcommand{\calL}{\mathcal{L}}
\newcommand{\calM}{\mathcal{M}}

\newcommand{\calS}{\mathcal{S}}

\newcommand{\calV}{\mathcal{V}}

\newcommand{\WBP}{\operatorname{WBP}}
\newcommand{\Leb}{\operatorname{Leb}}

\newcommand{\RMF}{\operatorname{RMF}}

\newcommand{\si}{\sigma}

\newcommand{\bla}{\big \langle}
\newcommand{\bra}{\big \rangle}

\newcommand{\ud}[0]{\,\mathrm{d}}

\newcommand{\id}[0]{\operatorname{id}}
\newcommand{\BMO}[0]{\operatorname{BMO}}
\newcommand{\supp}[0]{\operatorname{spt}}
\newcommand{\loc}[0]{\operatorname{loc}}

\newcommand{\eps}[0]{\varepsilon}

\newcommand{\Rad}[0]{\operatorname{Rad}}
\newcommand{\UMD}{\operatorname{UMD}}

\newcommand{\ch}[0]{\operatorname{ch}}

\newcommand{\calD}[0]{\mathcal{D}}

\newcommand{\wt}[1]{{\widetilde{#1}}}

\swapnumbers
\theoremstyle{plain}
\newtheorem{thm}[equation]{Theorem}
\newtheorem{lem}[equation]{Lemma}
\newtheorem{prop}[equation]{Proposition}
\newtheorem{cor}[equation]{Corollary}

\theoremstyle{definition}
\newtheorem{defn}[equation]{Definition}

\newtheorem{exmp}[equation]{Example}

\theoremstyle{remark}
\newtheorem{rem}[equation]{Remark}

\numberwithin{equation}{section}

\author{Francesco Di Plinio} \address[F.D.P.]{Department of Mathematics, Washington University in St. Louis, One Brookings Drive, \newline \indent St. Louis,  MO 63130-4899, USA}
\email{francesco.diplinio@wustl.edu}

 \author{Kangwei Li}
\address[Kangwei Li]
{Center for Applied Mathematics, Tianjin University, Weijin Road 92, 300072 Tianjin,\newline \indent China}
	\address{and}
\address{
BCAM, Basque Center for Applied Mathematics, Mazarredo 14, 48009
	Bilbao, Basque Country,\newline \indent   Spain}
\email{kli@tju.edu.cn}

\author{Henri Martikainen}
\address[H.M.]{Department of Mathematics and Statistics, University of Helsinki, P.O.B. 68, \newline \indent FI-00014 University of Helsinki, Finland}
\email{henri.martikainen@helsinki.fi}

\author{Emil Vuorinen}
\address[E.V.]{Centre for Mathematical Sciences, University of Lund, P.O.B. 118, 22100 Lund, Sweden}
\email{j.e.vuorin@gmail.com}

\title[Multilinear singular integrals on non-commutative $L^p$ spaces]{Multilinear singular integrals on non-commutative $L^p$ spaces}

\makeatletter
\@namedef{subjclassname@2010}{%
  \textup{2010} Mathematics Subject Classification}
\makeatother

\thanks{ F. Di Plinio has been
 partially supported by the National Science Foundation under the grants   NSF-DMS-1650810,  NSF-DMS-1800628 and NSF-DMS-2000510.
\\
  K. Li  was supported by Juan de la Cierva - Formaci\'on 2015 FJCI-2015-24547, by the Basque Government through the BERC
2018-2021 program and by Spanish Ministry of Economy and Competitiveness
MINECO through BCAM Severo Ochoa excellence accreditation SEV-2017-0718
and through project MTM2017-82160-C2-1-P funded by (AEI/FEDER, UE) and
acronym ``HAQMEC''.
\\
H. Martikainen was supported by the Academy of Finland through the grants 294840 and 306901, and by the three-year research grant 75160010 of the University of Helsinki.
He is a member of the Finnish Centre of Excellence in Analysis and Dynamics Research.
\\
E. Vuorinen was supported by the Academy of Finland through the grant 306901, by the Finnish Centre of Excellence in Analysis and Dynamics Research, and by
Jenny and Antti Wihuri Foundation.}

\begin{document}

\begin{abstract}
We  prove $L^p$ bounds for   the extensions of standard multilinear Calder\'on-Zygmund operators   to tuples of $\UMD$ spaces tied by  a natural product structure. The product can, for instance, mean the pointwise product in $\UMD$ function lattices, or the composition of operators in the  Schatten-von Neumann  subclass of the algebra of bounded operators on a Hilbert space. We do not require additional assumptions beyond $\UMD$ on each space -- in contrast to previous results, we e.g. show that the Rademacher maximal function property is not necessary. The obtained generality allows for novel applications. For instance, we prove new versions of fractional Leibniz rules via our results concerning the boundedness of multilinear singular integrals in non-commutative $L^p$ spaces.
Our proof techniques combine a novel  scheme of induction on the multilinearity index with dyadic-probabilistic techniques in the $\UMD$ space setting. 
\end{abstract}

\maketitle

\section{Introduction}
A Banach space  $X$ has the $\UMD$ property if any $X$-valued martingale difference sequence converges unconditionally in $L^p$ for some (equivalently, all) $p\in (1,\infty)$. Standard examples of $\UMD$ spaces are provided by the reflexive $L^p$ function spaces, as well as the reflexive Schatten-von Neumann subclasses $S^p$ of the algebra of bounded operators on a  Hilbert space.
The works by Burkholder \cite{Burk1} and Bourgain \cite{Bou} yield an alternative characterization: $X$ is a $\UMD$ space if and only  if singular integrals, in particular the Hilbert transform,  admit an $L^p(X)$-bounded extension. Such equivalence, albeit striking, is not so surprising when viewed from the modern dyadic-probabilistic perspective on singular integral operators. Indeed, 
Petermichl \cite{pet2000,Pet} realized that the Hilbert transform lies in the convex hull of certain dyadic operators akin to martingale transforms (the so-called dyadic shifts), while Hyt\"onen \cite{Hy} extended this representation to general singular integral operators of Calder\'on-Zygmund type, relying on a probabilistic construction. These results have roots in
the pioneering work of Figiel \cite{Fig} and on the probabilistic approach of Nazarov--Treil--Volberg to non-homogeneous $Tb$ theorems \cite{NTV}.

The theory of \emph{linear} singular integrals on Banach spaces, beyond its intrinsic interest, has historically been motivated by its interplay with several related areas, such as geometry of Banach spaces \cite{JX,KW1}, elliptic and parabolic regularity theory \cite{CH,Weis}, the theory of quasiconformal mappings \cite{GMSS}.  Furthermore, vector-valued bounds  may often be used in the pursuit of their  multi-parameter analogs \cite{HytPor08,Hyt2005}.

In this article, we are concerned with  Banach-valued extensions of \emph{multilinear} singular integral operators. 
A linear singular integral takes the general form
\begin{equation}\label{eq:SIO}
  Tf(x)=\int_{\R^d}K(x,y)f(y)\ud y,
\end{equation}
where different assumptions on the {\em kernel} $K$ lead to important classes of linear transformations arising across pure and applied analysis.
The term singular integral refers just to the underlying kernel structure -- a Calder\'on-Zygmund operator is a bounded singular
integral operator.
A heuristic model of an $n$-linear singular integral operator $T$ in $\R^d$ is then obtained by setting
$$
T(f_1,\ldots, f_n)(x) = U(f_1 \otimes \cdots \otimes f_n)(x,\ldots,x), \qquad x \in \R^d,\, f_i \colon \R^d \to \C,
$$
where $U$ is a linear singular integral operator in $\R^{nd}$. For the basic theory see e.g. Grafakos--Torres \cite{GT}.

Multilinear singular integrals arise naturally from applications to partial differential equations, complex function theory and ergodic theory, among others. Focusing on the results of greater significance for the present work, we mention that $L^p$ estimates for the fractional derivative  of a product, often referred to as \emph{fractional Leibniz rules},  are widely employed in the study of dispersive equations starting from the work of Kato and Ponce \cite{KP}, descend from the multilinear H\"ormander-Mihlin multiplier  theorem of Coifman-Meyer \cite{CM}. The bilinear Hilbert transform is a prime example of a \emph{modulation invariant} bilinear Calder\'on-Zygmund operator. It rose to prominence with Calder\'on's first commutator program, and has been featured as a model operator in the study of bilinear ergodic averages; the latter connection is expounded in e.g.\ \cite{DOP}.    Proving $L^p$ estimates for the bilinear Hilbert transform in the Lacey-Thiele framework \cite{LT1,LT2} involves a decomposition into  \emph{single trees}, which are essentially modulated bilinear Calder\'on-Zygmund operators. 

Vector-valued extensions of multilinear Calder\'on-Zygmund operators have mostly been studied within the more restrictive framework of $\ell^p$ spaces and function lattices. Boundedness of these extensions is classically obtained through weighted norm inequalities, more recently in connection with localized techniques such as sparse domination: see \cite{GM} and the more recent \cite{CUDPOULMS,LMMOV,Nieraeth} for a non-exhaustive  overview of their interplay. The paper \cite{DO}, by Y.\ Ou and one of us, contains a bilinear multiplier theorem which applies to certain non-lattice $\UMD$ spaces. The approach of \cite{DO} is  based on a localization of the $\UMD$-valued tent space norms, see for instance \cite{HNP}, within the Carleson embedding framework of Do and Thiele \cite{DoThiele}. The tent space techniques lead to the additional assumption of  $L^p$ estimates for a certain   analogue of the Hardy-Littlewood  maximal operator obtained by replacing uniform bounds with randomized, or $\mathcal R$-bounds, see e.g.\ \cite{Weis} for a definition. This assumption, usually referred to as the RMF property of  $X$, dates back  to the work of Hyt\"onen, McIntosh and Portal on the vector-valued Kato square root problem \cite{HMP}, and is in fact necessary for the $X$-valued Carleson embedding theorem to hold \cite{HK1}. 

 In this article, we obtain vector-valued extensions of multilinear singular integrals to tuples of $\UMD$ spaces tied by  a natural product structure, such as that of pointwise product in $\UMD$ function lattices or, more generally in fact, that of composition within the Schatten-von Neumann classes. We do not require additional conditions on the spaces involved -- in particular, we do not require the RMF property. Thus, we are able to extend multilinear Calrer\'on--Zygmund operators to natural tuples of non-commutative $L^p$ spaces -- a result which
 does not seem attainable via abstract theorems involving multilinear $\RMF$ type assumptions. A motivating corollary is a version of  the fractional Leibniz rule for products of  Schatten-von Neumann class-valued functions.

 In contrast to \cite{DO,HMP,HNP}, our techniques are dyadic-probabilistic:  a multilinear version of the representation theorem of Hyt\"onen \cite{Hy}, which appeared in the bilinear case in \cite{LMOV} by Y.\ Ou and three of us, reduces the problem to  the boundedness of the extensions of a class of multilinear dyadic model operators, namely \emph{paraproducts} and \emph{multilinear dyadic shifts} of arbitrary complexity. The novelty lies in how we treat these operators -- multilinearity poses significant problems in the vector-valued setup.
 
  We note that  $\UMD$-valued extensions of  bilinear, complexity zero dyadic shifts  have implicitly been treated in the work by Hyt\"onen, Lacey and Parissis  on the $\UMD$  dyadic model of the bilinear Hilbert transform \cite[Section 6]{HLP}. The simple approach of \cite{HLP} does not extend to either the higher complexity or the multilinear  cases. We tackle the $n$-linear case by inducting suitably on the linearity, which is made possible by associating to our $n$-tuples of $\UMD$ spaces a collection of  related $m$-tuples, $m<n$.
The framework is carefully designed to allow us to treat non-commutative theory. Moreover,
bilinear theory would not reveal all the difficulties and is, in fact, strictly easier -- a feature that is also present in our followup paper \cite{DLMV2} involving
operator-valued multilinear analysis.
Before providing further insights on the novelty of our proof techniques, and comparisons to previous approaches, we give the statements of our main results.

   \subsection{Main results}
We start by discussing a simpler question, where the current literature already has some restrictions that we can lift.
If $X$ is a Banach space and $T$ is an $n$-linear integral operator on $\R^d$ acting on $n$-tuples of functions in $L^\infty_c(\R^d)$, we may let $T$ act on $(L^\infty_c(\R^d) \otimes X) \times L_c^\infty(\R^d)\times \cdots \times L_c^\infty(\R^d) $ by\begin{equation}
\label{e:formal1}
\begin{split}
&T\left( f_1,f_2,\ldots, f_n\right) (x) = \sum e_{1,j} T(f_{1,j}, f_2,\ldots,f_n) (x), \qquad x\in \R^d, 
\\ &f_1= \sum   e_{1,j} f_{1,j}, \qquad f_{1,j}\in L^\infty_c(\R^d),\,e_{1,j}\in X.
\end{split}\end{equation}
A basic thing implied by our methods is that $n$-linear Calder\'on-Zygmund operators extend boundedly when applied to one $\UMD$-valued function and $n-1$ scalar functions, without any additional assumption on the $\UMD$ space. We send to Subsection \ref{sec:MultiSing} for the precise definition of an $ n$-linear Calder\'on-Zygmund operator.
This is the simplest complete multilinear analogue of Bourgain's $\UMD$ H\"ormander-Mihlin   multiplier theorem from \cite{Bou}; see also Weis \cite{Weis} and Hyt\"onen-Weis \cite{HW} for the operator-valued, non-translation invariant case.

In the bilinear, translation invariant, operator-valued  setting, a related result appeared in \cite[Corollary 1.2]{DO} under the assumption, known to be rather restrictive, that $X$ is a $\UMD$ space with the non-tangential Rademacher maximal function property \cite{HMP}. Theorem \ref{thm:X} shows, in particular,
 that the latter assumption is unnecessary. However, we formulate the following more general version to facilitate the discussion below regarding the somewhat special nature of bilinear theory.
 
 \begin{thm} \label{thm:X} Let $X_1, X_2, Y_3$ be $\UMD$ spaces with an associated product (a bounded bilinear operator)
$$
X_1 \times X_2 \to Y_3\colon (x_1, x_2) \mapsto x_1 x_2, \qquad |x_1x_2|_{Y_3} \le |x_1|_{X_1} |x_2|_{X_2}.
$$
Let $n\geq 2$ and $T$ be an $n$-linear Calder\'on-Zygmund operator on $\R^d$. The $n$-linear operator 
\begin{equation}
\begin{split}
&T\left( f_1,f_2,\ldots, f_n\right) (x) = \sum_{j_1, j_2} e_{1,j_1}e_{2,j_2} T(f_{1,j_1}, f_{2,j_2},f_3, \ldots,f_n) (x), \qquad x\in \R^d, 
\\ &f_1= \sum_{j_1}   e_{1,j_1} f_{1,j_1},\,  f_2= \sum_{j_2}   e_{2,j_2} f_{2,j_2} \qquad f_{1,j}, f_{2,j}\in L^\infty_c(\R^d),\,e_{1,j_1} \in X_1, \, e_{2,j_2} \in X_2,
\end{split}\end{equation}
extends to a bounded operator
\begin{align*}
T\colon L^{p_1}(\R^d; X_1) \times &L^{p_2}(\R^d; X_2) \times \prod_{k=3}^n L^{p_j}(\R^d) \to L^{q_{n+1}}(\R^d; Y_3), \\
&{ 1<p_k\leq \infty, \, \frac{1}{n}< q_{n+1} < \infty}, \,{\textstyle  \frac{1}{q_{n+1}}} = \sum_{k=1}^n {\textstyle \frac{1}{p_k}}.
\end{align*}
\end{thm}
The proof of this model case is an adaptation of the proof of Theorem \ref{thm:scal} with some additional observations regarding the bilinear case -- see
Remark \ref{rem:2VSn}.
This simpler result also showcases why the genuine $n$-linear theory that we formulate next is harder than bilinear theory: the $n$-linear theory requires us to exploit a more careful product setting so that we can run our inductive proof. We also note that at least in the basic case $X_1 = Y_3 = X$ and $X_2 = \C$, Theorem \ref{thm:X} can also be seen as a corollary of Theorem \ref{thm:scal} using Example \ref{ex:tensoralgebra}. It is simpler to just look at the proof, however.

Our main theorem concerns extensions of $n$-linear CZO operators $T$ to an $n$-tuple $X_1,\ldots, X_n$ of $\UMD$ Banach spaces lying in an enveloping algebra $\mathcal A$, allowing for a standard definition of (associative, not necessarily abelian) product $\mathcal A \times \mathcal A \to \mathcal A$.  We refer to these configurations as $\UMD$ H\"older tuples if certain conditions are in place,  in particular, if the $n$-tuples are associated with suitable collections of related $m$-tuples, $m<n$.
If each $X_j$ is a subspace of $\mathcal A$, and $f_k\in L^\infty_c(\R^d) \otimes X_k$ for $1\leq k\leq n$, we may define the extension of a scalar integral operator by
\begin{equation}
\label{e:formal2}
\begin{split}
&T\left( f_1, \ldots, f_n\right) (x) = \sum_{j_1, \ldots, j_n}  T(f_{1,j_1}, \ldots,f_{n,j_n}) (x) \prod_{k=1}^n e_{k,j_k}, \qquad x\in \R^d, 
\\ &f_{k}= \sum_{j_k}  e_{k,j_k} f_{k,j_k}, \qquad f_{k,j_k}\in L^\infty_c(\R^d),\,e_{k,j_k}\in X_k.
\end{split}\end{equation} 
The abstract setup   is  developed in Section \ref{sec:UMDHT}. For expository purposes, herein  we provide a statement in a rather general concrete case of a $\UMD$ H\"older tuple. In the statement,  we denote  by $L^p(\mathcal M)$ the non-commutative $L^p$ spaces associated to   a von Neumann algebra $ \mathcal M $ endowed with a normal, semifinite, faithful trace $\tau$.
\begin{thm} \label{thm:Xj}
Let $ \mathcal M$ be a von Neumann algebra   endowed with a normal, semifinite, faithful trace. For $s=1,\ldots, S$, let $(M_s,\mu_s)$ be measure spaces and for $s=0,\ldots, S$ let 
 \begin{equation}
1<p^s_{1},\ldots p^s_{n},  q^s_{n+1}<\infty,\qquad  {  \frac{1}{q_{n+1}^s}} = \sum_{k=1}^n {  \frac{1}{p_k^s}}
\label{e:longtup}
\end{equation}
be Banach H\"older tuples. Let
\begin{equation}
\label{e:UMDsp1}
\begin{split}
&X_k=L^{p_k^S} (M_S,\mu_s; L^{p_k^{S-1}} (M_{S-1},\mu_{S-1};\cdots L^{p_k^{1}} (M_{1},\mu_{1}; L^{p_k^{0}}(\mathcal M))\cdots),\qquad k=1,\ldots,n,  \\ &Y_{n+1}=L^{q_{n+1}^S} (M_S,\mu_S; L^{q_{n+1}^{S-1}} (M_{S-1},\mu_{S-1};\cdots L^{q_{n+1}^{1}} (M_{1},\mu_{1}; L^{q_{n+1}^{0}}(\mathcal M))\cdots).
\end{split}
\end{equation}
 The $n$-linear operator \eqref{e:formal2} extends to a bounded operator
\[
\begin{split}
&
T: \prod_{k=1}^n L^{p_k}(\R^d; X_k)  \to L^{q_{n+1}}(\R^d; Y_{n+1}), \quad { 1<p_k\leq \infty, \, \frac{1}{n}< q_{n+1} < \infty}, \,{\textstyle  \frac{1}{q_{n+1}}} = \sum_{k=1}^n {\textstyle \frac{1}{p_k}},
\\ &T: \prod_{k=1}^n L^{1}(\R^d; X_k)  \to L^{\frac{1}{n},\infty}(\R^d; Y_{n+1}).
\end{split}
\]
 In fact, we have the stronger estimate
\begin{equation}\label{eq:eq1}
\begin{split}
&|\langle T(f_1, \ldots, f_n ), f_{n+1}\rangle | \lesssim  \left\|\mathrm{M}\left(|f_1|_{X_1} ,\ldots,|f_n|_{X_n}, |f_{n+1}|_{Y_{n+1}^*}\right)\right\|_{1},
\\ &\mathrm{M}(g_1,\ldots, g_{n+1})(x) \coloneqq \sup_{x\in Q } \prod_{j=1}^{n+1}\langle |g_j| \rangle_Q, \qquad \langle g \rangle_Q {\coloneqq} \frac{1}{|Q|} \int_Q g.
\end{split}
\end{equation}
The estimate \eqref{eq:eq1} is equivalent to a certain sparse bound, see Remark \ref{r:sparsemf}.
\end{thm}
We send to Subsection \ref{ss:sparse} and to the references \cite{CDPOMRL,CDPOBP} for    more details on sparse bounds and to \cite{LMMOV,LMO} for a survey of the   weighted inequalities that may be derived as a consequence.
%We remark that, in the setup of Theorem \ref{thm:Xj}, the product appearing in \eqref{e:formal2} is naturally interpreted within the algebra $\mathcal A=L^0(\mathscr M)$ of measurable operators associated to the semicommutative von Neumann algebra  \cite{PX} \[\mathscr
%M=\left(\bigotimes_{j=1}^A L^\infty(M_j,\mu_j)\right) \otimes \mathcal M\] with the natural trace. If $\mathcal M$ is commutative, and therefore $X_1,\ldots, X_n,Y_{n+1}$ are Banach function lattices, \eqref{e:formal2} is just the pointwise product. 

 Theorem \ref{thm:Xj} is obtained as a corollary of Theorem \ref{thm:scal} using Example \ref{ex:thmXj}. However, we remark that, at least to the best of the authors' knowledge,  the spaces \eqref{e:UMDsp1} encompass all known \emph{examples} of $\UMD$ Banach spaces. We further remark that the mixed norm structure of the spaces \eqref{e:UMDsp1} prevents from using purely non-commutative tools, as \eqref{e:UMDsp1} may be interpreted as semi-commutative spaces only if $p_k^s$ does not vary with $s$ for all $1\leq k \leq n$; on the other hand, \eqref{e:UMDsp1} are not $\UMD$ lattices, so that Theorem \ref{thm:Xj} is out of reach of purely lattice-type techniques. 
 
  Theorems \ref{thm:X} and \ref{thm:Xj} can be used to deduce certain weighted multilinear  Leibniz rules in the $\UMD$-valued and non-commutative setting. For simplicity of notation, we particularize the statements to the bilinear, unweighted, non-endpoint case for the homogeneous fractional derivative $D^s f=\mathcal F^{-1}(|\xi|^s \widehat f(\xi))$, in the setting of Theorem \ref{thm:X}.  A variety of formulations may be found e.g.\ in the article by Grafakos and Oh \cite{GrafOh}.
\begin{cor}[Fractional Leibniz rules in $\UMD$ spaces] \label{cor:leib1} Let $X_1,X_2,Y_3$ be  $\UMD$ spaces as in the statement of Theorem \ref{thm:X}. For all sufficiently smooth $f_1: \R^d\to X_1, f_2: \R^d \to X_2$  there holds 
\[
\left\|D^s(f_1f_2) \right\|_{L^{q_3}(\R^d; Y_3)} \lesssim \left\|D^sf_1  \right\|_{L^{p_1}(\R^d; X_1)}\left\|f_2  \right\|_{L^{p_2}(\R^d; X_2)} + \left\| f_1  \right\|_{L^{r_1}(\R^d; X_1)}\left\|D^sf_2  \right\|_{L^{r_2}(\R^d; X_2)}
\]
whenever $s>d$ and 
\begin{equation}
\label{e:leibtup}
1<p_1,p_2,r_1,r_2 \leq \infty, \  \frac12 <q_3<\infty, \  \textstyle \frac{1}{q_3}=\frac{1}{p_1}+\frac{1}{p_2} = \frac{1}{r_1}+\frac{1}{r_2}.
\end{equation}
\end{cor} 
Corollary \ref{cor:leib1} appears to be the first instance of a Leibniz type rule in the full vector-valued setting, with no additional assumptions on the $\UMD$ spaces involved.
We have not strived for optimality  of the range for the fractional exponent $s$. While the range obtained in  Corollary \ref{cor:leib1} is wider than what would follow from results of  Coifman-Meyer type, see \cite[Remark 1]{GrafOh},  the extension to the sharp range    $s>\max \Big\{0, d\left(\frac{1}{q_3}-1\right)\Big\}$ requires bilinear estimates for kernels which fail to be of the standard CZ type considered herein. Such estimates are carried out e.g. in \cite{GrafOh}: their extension to the full vector-valued setting  is left for future work.  \begin{proof}[Proof of Corollary \ref{cor:leib1}] 
We follow the beginning of the  proof of  \cite[Theorem 1]{GrafOh}. The estimate we seek is reduced to a bound for the $\UMD$-valued extension of  three different bilinear paraproducts (meaning suitable parts of a Littlewood--Paley decomposition of a product of functions -- not in the exact sense as we use the word in connection with dyadic model operators). We note that the symbol of the high-low paraproducts $\Pi_1$ and $\Pi_2$ is  of Coifman-Meyer type; therefore $\Pi_1,\Pi_2$ are bilinear CZO operators as defined in Subsection \ref{sec:MultiSing} and  Theorem \ref{thm:Xj} applies directly. The high-high term $\Pi_3$ is a bilinear integral  operator with kernel 
\[
K(x, y_1,y_2) = \sum_{m\in \mathbb Z} \int_{\R^d} 2^{3md} \phi_s(2^m (u-x))\psi(2^m (u-y_1))\psi(2^m (u-y_2)) \, \mathrm{d} u
\]
where $\psi$ is a Schwartz function whose Fourier transform $\Psi$ is supported in an annular region around the origin and $\phi_s=D^s\phi$ for some Schwartz function $\phi$ such that its Fourier transform has compact support containing $0$, so that
\[
|\phi_s(x)|   \lesssim (1+|x|)^{-(d+s)}, \qquad x\in \R^d.
\]
As $s > d$ for us, this implies that  $\Pi_{3}$ is a  bilinear CZO operator with a kernel $K$ satisfying  
\[
\|\Pi_{3}\|_{L^{3}\times L^3 \to L^\frac32} + \|K\|_{\mathrm{CZ}_{(s-d)/2}} \lesssim 1,
\]
where $\| K \|_{\operatorname{CZ}_\alpha}$ is the kernel constant defined in the beginning of Section \ref{sec:MultiSing}.
The required bounds for $\Pi_3 $ follow from an application of Theorem \ref{thm:X}. 
\end{proof}

\subsection{Proof techniques and novelties}
A basic example of an $n$-linear dyadic shift operator of complexity zero on $\R$, in adjoint form, is
\[
(f_1,\ldots,f_{n+1}) \mapsto
\sum_{m\in \mathbb Z}\eps_m \int   \Big(\prod_{k\in C} \Delta_m f_k (x)\Big)   \Big(\prod_{k\in N} E_m f_k (x)\Big) \, \mathrm{d} x
\]
where $\eps_m$ are bounded coefficients, and   $E_m$ and $\Delta_m$ respectively indicate the conditional expectation on the $m$-th dyadic filtration and the corresponding martingale difference, $C \cap N = {\varnothing}$ and $C\cup N=\{1,\ldots,n+1\}$, with the key feature that the cardinality of the \emph{cancellative} indices $C$ is always at least 2. We approach $\UMD$-valued extensions of the above forms to $(n+1)$-tuples of $\UMD$ spaces via a novel induction argument, aimed at reducing the cardinality of the set of \emph{non-cancellative} indices $N$ and the linearity of the shift $n$ at the same time. The induction relies upon a certain structure of the tuples involved, which is most easily described in the bilinear, $n=2$, case.  Loosely speaking, we consider $\UMD$ spaces $X_1,X_2,X_3$ endowed with a linear functional  $\tau$ defined on all  products $e_1e_2e_3$, $e_j\in X_j$, with the property that
\[
\|e_1\|_{X_1} \sim \sup_{|e_2|_{X_2}=|e_3|_{X_3}=1}
|\tau(e_1e_2 e_3)| 
\]
and the same holds for all permutations of $X_1,X_2,X_3$.   In combination with the martingale  decoupling inequality of McConnell \cite{Mc}  and Hyt\"onen \cite{Hy2}, and Stein's inequality in $\UMD$ spaces, this structure allows to reduce a trilinear shift form on $X_1,X_2,X_3$ where, say,   $1\in C$ and $2\in N$, to a bilinear shift form on $X_1,X_1^*$, where both indices are cancellative, and whose boundedness is known from the $\UMD$ character of $X_1$. 
The induction is crucial in the $n$-linear case to allow a repeated use of Stein's inequality.

We remark here that the martingale decoupling has been previously used by H\"anninen and Hyt\"onen \cite{HH} in the proof of a  $T1$ theorem for linear singular integrals on $\UMD$ spaces with operator-valued kernels, providing among other results a non-translation invariant analogue of Weis's theorem \cite{Weis}. The multilinear operator-valued theory, together with a related representation theorem, is the object of   forthcoming work by the authors \cite{DLMV2}.

\subsection*{Acknowledgments} The authors would like to warmly thank Yumeng Ou for fruitful discussions on the subject of multilinear UMD-valued singular integrals. F.\ Di Plinio is grateful to Ben Hayes and Vittorino Pata for enlightening exchanges on factorization in noncommutative $L^p$ spaces.

\section{Definitions and preliminaries} 
\subsection{Vinogradov notation}
We write $A \lesssim B$ if $A \le CB$ for some absolute constant $C$. The constant $C$ can at least depend on the dimensions of the appearing Euclidean spaces, on integration exponents, on the degree of linearity of the multilinear operators, 
and on various Banach space constants.
We use the notation $A \sim B$ if $B \lesssim A \lesssim B$.

\subsection{Dyadic notation}\label{sec:randomlattice}
Let $\calD_0$ be the dyadic lattice in $\R^d$, defined by
$$
\calD_0=\{2^{-k}([0,1)^d+m) \colon k \in \Z, m \in \Z^d\}.
$$
We recall the random dyadic grids of Nazarov--Treil--Volberg, see for example \cite{NTV}. 
The version we use here is from \cite{Hy2}.
Let $\Omega=(\{0,1\}^d)^{\Z}$ and let $\bbP$ be the natural probability measure on $\Omega$ such that the coordinates are independent and uniformly distributed on $\{0,1\}^d$. If
$Q \in \calD_0$ and $\omega=(\omega_k)_{k\in \Z} \in \Omega$, we set
$$
Q+\omega{\coloneqq}Q+\sum_{k \colon 2^{-k} < \ell(Q)} \omega_k2^{-k}.
$$
The random dyadic lattice $\calD_\omega$ on $\R^d$ is defined by
$
\calD_\omega
= \{Q+\omega \colon Q \in \calD_0\}.
$
By a dyadic lattice $\calD$ we mean that $\calD=\calD_\omega$ for some $\omega$.

Let $X$ be a Banach space. If $p \in (0, \infty]$ we denote by $L^p(X)=L^p(\R^d;X)$ the usual
Bochner space of $X$-valued functions $f \colon \R^d \to X$. Let $\calD$ be a dyadic lattice. 
Suppose $Q \in \calD$ and $f \in L^1_{\loc}(X)$ (the set of locally integrable functions).
We use the following notation:
\begin{itemize}
\item The side length of $Q$ is denoted by $\ell(Q)$;
\item $\ch (Q)$ consists of those $Q' \in \calD$ such that $Q' \subset Q$ and $\ell(Q')=\ell(Q)/2$;
\item If $k \in \Z$, $k \ge 0$, then $Q^{(k)}$ denotes the cube $R \in \calD$ such that $Q \subset R$ and 
$2^k\ell(Q) =\ell(R)$;
\item The average of $f$ over $Q$ is $\langle f \rangle_Q = \frac{1}{|Q|} \int_Q f \ud x$,
and we also write $E_Q f=\langle f \rangle_Q 1_Q$;
\item The martingale difference $\Delta_Q f$ is $\Delta_Q f= \sum_{Q' \in \ch (Q)} E_{Q'} f - E_{Q} f$;
\item For $k \in \Z$, $k \ge 0$,  define 
$$
\Delta_Q^k f=\sum_{\substack{R \in \calD \\ R^{(k)}=Q}} \Delta_{R} f 
\quad \text{and} \quad
E_Q^k f=\sum_{\substack{R \in \calD \\ R^{(k)}=Q}} E_{R} f.
$$
\end{itemize}

\subsubsection*{Haar functions}
When $Q \in \calD$ we denote by $h_Q$ a cancellative $L^2$ normalized Haar function. This means the following.
Writing $Q = I_1 \times \cdots \times I_d$ we can define the Haar function $h_Q^{\eta}$, $\eta = (\eta_1, \ldots, \eta_d) \in \{0,1\}^d$, by setting
\begin{displaymath}
h_Q^{\eta} = h_{I_1}^{\eta_1} \otimes \cdots \otimes h_{I_d}^{\eta_d}, 
\end{displaymath}
where $h_{I_i}^0 = |I_i|^{-1/2}1_{I_i}$ and $h_{I_i}^1 = |I_i|^{-1/2}(1_{I_{i, l}} - 1_{I_{i, r}})$ for every $i = 1, \ldots, d$. Here $I_{i,l}$ and $I_{i,r}$ are the left and right
halves of the interval $I_i$ respectively. If $\eta \ne 0$ the Haar function is cancellative: $\int h_Q^{\eta} = 0$. We usually exploit notation by suppressing the presence of $\eta$,
and simply write $h_Q$ for some $h_Q^{\eta}$, $\eta \ne 0$. 

Notice that if $f \in L^1_{\loc}(X)$, then $\Delta_Q f = \sum_{\eta \ne 0} \langle f, h_{Q}^{\eta}\rangle h_{Q}^{\eta}$, or suppressing the $\eta$ summation, $\Delta_Q f = \langle f, h_Q \rangle h_Q$.
Here $\langle f, h_Q \rangle = \int f h_Q$.

\subsection{Definitions and properties related to Banach spaces}\label{sec:BanachProperties}
An extensive treatment of Banach space theory is given in the books \cite{HNVW1, HNVW2} by Hyt\"onen, van Neerven, Veraar and Weis.

We say that $\{\varepsilon_k\}_k$ is a \emph{collection of independent random signs}, where $k$ runs over some index set, if there exists a probability space $(\calM, \mu)$ so that $\varepsilon \colon \calM \to \{-1,1\}$, 
$\{\varepsilon_k\}_k$ is independent and
$\mu(\{\varepsilon_k=1\})=\mu(\{\varepsilon_k=-1\})=1/2$. 
Below, $\{\varepsilon_k\}_k$ will always denote a collection of independent random signs.

Suppose $X$ is a Banach space. We denote the underlying norm by $| \cdot |_X$.
The Kahane-Khintchine inequality says that
for all $x_1,\dots, x_M \in X$ and $p,q \in (0,\infty)$ there holds that
\begin{equation}\label{eq:KK}
\Big(\E \Big | \sum_{m=1}^M \varepsilon_m x_m \Big |_X^p \Big)^{1/p}
\sim \Big(\E \Big | \sum_{m=1}^M \varepsilon_m x_m \Big |_X^q \Big)^{1/q}.
\end{equation}
We also denote
$$
\|(x_m)\|_{\Rad(X)} {\coloneqq} \Big(\E \Big | \sum \varepsilon_m x_m \Big |_X^2 \Big)^{1/2}.
$$
The Kahane contraction principle says that if $(a_m)_{m=1}^M$ is a sequence of scalars and $p \in (0, \infty]$, then
\begin{equation}\label{eq:KCont}
\Big( \E \Big | \sum_{m=1}^M \varepsilon_m a_m x_m \Big |_{X}^p \Big)^{1/p}
\lesssim \max |a_m| \Big( \E \Big | \sum_{m=1}^M \varepsilon_m x_m \Big |_{X}^p \Big)^{1/p}.
\end{equation}
Actually, if $p \in [1, \infty]$ and $a_m \in \R$, then \eqref{eq:KCont} holds with ``$\le$'' in place of ``$\lesssim$'', 
see \cite{HNVW1} for more details.

A Banach space $X$ is said to be a $\UMD$ \emph{space} if for all $p \in (1,\infty)$, all
$X$-valued $L^p$-martingale difference sequences $(d_j)_{j=1}^k$ and signs $\epsilon_j \in \{-1,1\}$ 
there holds that
\begin{equation}\label{eq:UMDDef}
\Big\| \sum_{j=1}^k \epsilon_j d_j \Big \|_{L^p(X)} 
\lesssim \Big\| \sum_{j=1}^k d_j \Big \|_{L^p(X)}. 
\end{equation}
Here the $L^p(X)$-norm is with respect to the measure space where the martingale differences are defined. 
If the estimate \eqref{eq:UMDDef} holds for one $p_0\in (1,\infty)$, then it
holds for all $p \in (1, \infty)$.

%\subsubsection*{Random sums and duality}
%For the definition of type and cotype of a Banach space the reader can e.g. consult the section 7 of the book \cite{HNVW2}. 
%All UMD spaces have non-trivial type.
%For the next lemma see Section 7.4.f in \cite{HNVW2}.
%\begin{lem}\label{lem:randomduality}
%Let $X$ be a Banach space with non-trivial type and let $F \subset X^*$ be a closed subspace of $X^*$ which is norming for $X$. Let $p \in (1,\infty)$. Then
%for all finite sequences $e_1, \ldots, e_N \in X$ we have
%$$
%\Big( \E \Big| \sum_{i=1}^N \varepsilon_i e_i \Big|_X^p \Big)^{1/p} \lesssim \sup \Big\{ \Big| \sum_{i=1}^N \langle e_i, e_i^*\rangle \Big|\Big\},
%$$
%where the supremum is taken over all choices $(e_i^*)_{i=1}^N$ in $F$ such that
%$$
%\Big( \E \Big| \sum_{i=1}^N \varepsilon_i e_i^* \Big|_{X^*}^{p'} \Big)^{1/{p'}} \le 1. 
%$$
%The converse inequality trivially holds with a constant $1$.
%\end{lem}

A version for $\UMD$-valued functions of Stein's inequality concerning conditional expectations is due to  Bourgain. For a proof, see for example
 \cite[Theorem 4.2.23]{HNVW1}. For our purposes we formulate the estimate in the following way.
Suppose $X$ is a $\UMD$ space and let $\calD \subset \R^d$ be a dyadic lattice. Suppose
that for each $Q \in \calD$ we have a function $f_Q \in L^1_{\loc}(X)$ supported in $Q$ (such that only 
finitely many of them are non-zero). Then for all $p \in (1, \infty)$ there holds that
\begin{equation}\label{eq:SteinUMD}
\E \Big \| \sum_{Q \in \calD} \varepsilon_Q \langle f_Q \rangle_Q 1_Q \Big \|_{L^p(X)}
\lesssim \E \Big \| \sum_{Q \in \calD} \varepsilon_Q  f_Q \Big \|_{L^p(X)}.
\end{equation}

\subsubsection*{The decoupling inequality}
We record a special case of the decoupling estimate \cite[Theorem 6]{HH} by H\"anninen--Hyt\"onen. These
decoupling estimates originate from McConnell \cite{Mc}, but see also Hyt\"onen \cite{Hy2}.

Let $\calD$ be a dyadic lattice in $\R^d$ and $Q \in \calD$. 
Let $\calV_Q$ be the probability measure space $\calV_Q=(Q, \Leb(Q), |Q|^{-1} \ud x \lfloor Q)$,
where $\Leb(Q)$ is the set of Lebesgue measurable subsets of $Q$ and 
$|Q|^{-1} \ud x \lfloor Q$ is the normalized Lebesgue measure restricted to $Q$.
Define the product probability space $\calV= \prod_{Q \in \calD} \calV_Q$,
and let $\nu$ be the related measure. If $y \in \calV$, we denote the coordinate related to $Q \in \calD$
by $y_Q$.

Suppose $X$ is a $\UMD$ space, $p \in (1, \infty)$ and $f \in L^p(X)$.
Let $k \in \{0,1,2, \dots\}$ and $j \in \{0, \dots, k\}$.
Define $\calD_{j,k} \subset \calD$
by 
\begin{equation}\label{eq:SubLattice}
\calD_{j,k}=\{Q \in \calD \colon \ell(Q)=2^{m(k+1)+j} \text{ for some } m \in \Z\}.
\end{equation}
\cite[Theorem 6]{HH} implies that
\begin{equation}\label{eq:DecEst}
\int_{\R^d} \Big | \sum_{Q \in \calD_{j,k}} \Delta^l_Q f(x)  \Big |_X^p \ud x
\sim \E \int_{\R^d} \int_{\calV} \Big | \sum_{Q \in \calD_{j,k}} \varepsilon_Q 1_Q(x)\Delta^l_Q f(y_Q)  \Big |_X^p \ud \nu(y) \ud x
\end{equation}
for any $l \in \{0, 1, \ldots, k\}$.
The point of dividing to the subcollections $\calD_{j,k}$ is that now $\Delta^l_Q f$ is constant on
every $Q' \in \calD_{j,k}$ such that $Q' \subsetneq Q$, which is required by the decoupling theorem (together with the fact that
$\int \Delta^l_Q f = 0$ and $\supp\Delta^l_Q f \subset Q$).

%\subsubsection*{Pythagoras' theorem} 
%A collection $\mathcal{S}$ of cubes in $\R^d$ is said to be \emph{$\eta$-sparse} (or just sparse), $0<\eta<1$, if for any $Q\in\mathcal{S}$ there exists $E_Q\subset Q$ so that $|E_Q|>\eta|Q|$ and $\{E_Q \colon Q\in\mathcal{S}\}$ are pairwise disjoint.

%Let $\calD$ be a dyadic lattice, $\calS \subset \calD$ be sparse and $X$ be a Banach space. 
%Suppose that for every $S \in \calS$ we have a function $f_S \colon \R^d \to X$ such that 
%$f_S$ is supported in $S$,  $\int f_S \ud x=0$ and $f_S$ is constant on those $S' \in \calS$ such that $S' \subsetneq S$.
%Then, Lemma 4 from \cite{HH} gives that
%\begin{equation}\label{eq:Pythagoras}
%\Big \| \sum_{S \in \calS} f_S \Big \|_{L^p(X)}^p
%\sim \sum_{S \in \calS} \| f_S \|_{L^p(X)}^p.
%\end{equation}

\subsection{Multilinear singular integrals and model operators}\label{sec:MultiSing}
A function
$$
K \colon \R^{d(n+1)} \setminus \Delta \to \C, \qquad \Delta=\{x = (x_1, \dots, x_{n+1}) \in\R^{d(n+1)} \colon x_1=\dots =x_{n+1}\},
$$
is called an $n$-linear basic kernel if for some $\alpha \in (0,1]$ and $C_{K} < \infty$ it holds that
$$
|K(x)| \le \frac{C_{K}}{\Big(\sum_{m=2}^{n+1} |x_1-x_m|\Big)^{dn}},
$$
and for all $j \in \{1, \dots, n+1\}$ it holds that
$$
|K(x)-K(x')| \le C_{K} \frac{|x_j-x_j'|^{\alpha}}{\Big(\sum_{m=2}^{n+1} |x_1-x_m|\Big)^{dn+\alpha}}
$$
whenever $x=(x_1, \dots, x_{n+1}) \in\R^{d(n+1)} \setminus \Delta$ and 
$x'=(x_1, \dots, x_{j-1},x_j',x_{j+1},\dots x_{n+1}) \in\R^{d(n+1)}$ satisfy
$$
|x_j-x_j'| \le 2^{-1} \max_{2 \le m \le n+1} |x_1-x_m|.
$$
The best constant $C_K$ is called $\| K \|_{\operatorname{CZ}_\alpha}$.

An $n$-linear operator $T$ defined on a suitable class of functions (e.g. on the linear combinations of cubes)
is an $n$-linear \emph{singular integral operator (SIO)} with an associated kernel $K$, if we have
$$
\langle T(f_1, \dots, f_n), f_{n+1} \rangle=
\int_{\R^{d(n+1)}}  K(x_{n+1},x_1, \dots, x_n)\prod_{j=1}^{n+1} f_j(x_j) \ud x
$$
whenever $\supp f_i \cap \supp f_j = {\varnothing}$ for some $i \not= j$.

We say that $T$ is an $n$-linear \emph{Calder\'on--Zygmund operator (CZO)} if the following conditions hold:
\begin{itemize}
\item $T$ is an $n$-linear SIO.
\item We have that for all $m \in \{0, \ldots, n\}$ there holds that
$$
\| T^{m*}(1, \ldots, 1) \|_{\BMO} {\coloneqq} \sup_{\calD}
\sup_{K_0 \in \calD} \Big( \frac{1}{|K_0|} \sum_{\substack{K \in \calD \\ K \subset K_0}} |\langle T^{m*}(1, \ldots, 1), h_K\rangle|^2 \Big)^{1/2} < \infty,
$$
where the first supremum is taken over all dyadic lattices $\calD$.
Here $T^{0*} {\coloneqq} T$, $T^{m*}$ denotes the $m$th adjoint of T for $m \in \{1,\ldots,n\}$,
and the pairings $\langle T^{m*}(1, \ldots, 1), h_K\rangle$ have a standard $T1$ type definition with the aid of the kernel $K$.
\item We have that
$$
\|T\|_{\WBP} {\coloneqq} \sup_{\calD} \sup_{Q \in \calD}
|Q|^{-1} |\langle T(1_Q, \ldots, 1_Q), 1_Q\rangle| < \infty.
$$
\end{itemize}
An SIO $T$ is a CZO if and only if
\begin{equation}\label{eq:CZOBdd}
\| T(f_1, \dots, f_n) \|_{L^{q_{n+1}}(\R^d)}
\lesssim \prod_{m=1}^n \| f_m\|_{L^{p_m}(\R^d)}
\end{equation}
for \emph{some} (equivalently for all) exponents $p_1, \dots, p_n \in (1, \infty)$, $q_{n+1} \in (1/n, \infty)$ satisfying $\sum_{m=1}^n 1/p_m=1/q_{n+1}$.
While such a $T1$ theorem is well-known (see e.g. \cite{DLMV2, GT, LMOV}), we will need a very precise   version of this called a dyadic representation theorem. To this end, we need some definitions.

Let $k=(k_1, \dots, k_{n+1})$, $0 \le k_i \in \Z$, and let $\calD$ be a dyadic lattice in $\R^d$.
An operator $S = S_{\calD}^k$ is called an $n$-linear dyadic shift if it has the form
\begin{equation} \label{e:shift}
S(f_1,\dots,f_n)
=\sum_{K \in \calD} A_K(f_1, \ldots, f_n),
\end{equation}
where
$$
A_K(f_1, \ldots, f_n) = 
\sum_{\substack{Q_1, \dots, Q_{n+1} \in \calD \\ Q_j^{(k_j)}=K}}
a_{K,(Q_j)}\prod_{j=1}^n \langle f_j, \wt h_{Q_j} \rangle \wt h_{Q_{n+1}}.
$$
Here $a_{K,(Q_j)} = a_{K, Q_1, \ldots ,Q_{n+1}}$ is a scalar satisfying the normalization
$$
|a_{K,(Q_j)}| \le \frac{\prod_{j=1}^{n+1} |Q_j|^{1/2}}{|K|^{n}},
$$
and there exist two indices $j_0,j_1 \in \{1, \ldots, n+1\}$, $j_0 \not =j_1$, so that $\wt h_{Q_{j_0}}=h_{Q_{j_0}}$, $\wt h_{Q_{j_1}}=h_{Q_{j_1}}$ and $\wt h_{Q_j}=h_{Q_j}^0$ if $j \not \in \{j_0, j_1\}$.

An $n$-linear dyadic paraproduct $\pi = \pi_{\calD}$ also has $n+1$ possible forms, but there is no complexity (the $k = (k_1, \ldots, k_{n+1})$) associated
to them. One of the forms is
$$
\pi(f_1,\ldots,f_n)
=\sum_{K \in \calD} a_K \prod_{j=1}^n \langle f_j \rangle_K h_K,
$$
where the coefficients satisfy the BMO condition
\begin{equation}\label{eq:BMOCondition}
\sup_{K_0 \in \calD} \Big( \frac{1}{|K_0|} \sum_{\substack{K \in \calD \\ K \subset K_0}} |a_K|^2 \Big)^{1/2}
\le 1.
\end{equation}
This is the paraproduct associated with the tuple $(1_K / |K|, \ldots, 1_K/|K|, h_K)$, and in the remaining $n$ alternative
forms the $h_K$ is in a different position.

We call shifts and paraproducts \emph{dyadic model operators (DMOs)}.
Suppose $T$ is an $n$-linear Calder\'on-Zygmund operator in $\R^d$ related to a kernel $K$. If $f_1, \ldots, f_{n+1}$ are, say, $L^{n+1}(\R^d)$ functions, then the representation theorem states that
\begin{equation}\label{eq:repthm}
\langle T(f_1, \dots, f_n), f_{n+1} \rangle
=C_T \E_\omega \sum_{k_1, \ldots, k_{n+1} =0}^\infty  \sum_u 2^{-\max k_i \alpha/2} 
\langle U^k_{\calD_\omega,u} (f_1, \dots, f_n), f_{n+1} \rangle.
\end{equation}
Here \begin{align*}
|C_T| \lesssim \sum_{m=0}^n \| T^{m*}(1, \ldots, 1) \|_{\BMO} &+ \|T\|_{\WBP} + \| K \|_{\operatorname{CZ}_\alpha} \\&\lesssim
\| T \|_{L^{n+1} \times \cdots \times L^{n+1} \to L^{(n+1)/n}}+ \| K \|_{\operatorname{CZ}_\alpha},
\end{align*}
$\alpha$ is the parameter in the H\"older continuity assumptions of the kernel of $T$,
and the sum over $u$ is finite, say, over $u=1,2, \dots, C(n,d)$.
If $\max {k_i}>0$, then $U^k_{\calD_\omega,u}$ is some dyadic shift $S_{\calD_\omega}^k$ of complexity $k$ 
with respect to the lattice $\calD_{\omega}$. If $\max k_i=0$, then $U^k_{\calD_\omega,u}$ is a shift of complexity zero or 
a paraproduct. In this sense, a CZO $T$ can be represented using DMOs. For $n=2$, a proof of this result is given by
three of us and Y. Ou in \cite{LMOV}. The $n$-linear case for general $n$, which requires certain modifications, is \cite[Theorem 6.3]{DLMV2}.
The reference \cite[Theorem 6.3]{DLMV2} is a more general theorem involving operator-valued CZOs. We note that the additional assumptions related to the
operator-valued setup, such as the $\operatorname{RMF}$ assumption, concern only the estimation of the model operators. They are not needed
for the above stated structural theorem, which has essentially the same proof in the scalar-valued and operator-valued settings.

As DMOs satisfy $L^p$ estimates in the full expected range of exponents, the $T1$ theorem follows from the representation theorem.
Our main task in this paper will be to prove $L^p$-bounds for the  extensions of $n$-linear DMOs to suitably defined tuples of $\UMD$ spaces, which we term $\UMD$ H\"older tuples and define in the subsequent section.

\section{UMD H\"older tuples and the boundedness of multilinear SIOs} \label{sec:UMDHT}
Throughout this section, and the remainder of the article, we make use of the following notational conventions. For $m \in \N$ we write $\calJ_m{\coloneqq} \{1, \dots, m\}$ and denote the set of permutations of $\mathcal J\subset \mathcal J_m$ by $\Sigma(\mathcal J)$.   We simply write $\Sigma(m)$ in place of $\Sigma(\calJ_m)$. We say that $p_1,\ldots,p_m$ is a H\"older tuple of exponents if
\begin{equation}
\label{e:holder}
1<p_1,\ldots, p_{m}<\infty, \qquad \sum_{j=1}^{m} \textstyle \frac{1}{p_j} = 1.
\end{equation} 

\subsection{UMD H\"older tuples.}

The notion of UMD H\"older tuple involves fixing    an associative algebra $\mathcal A$ over $\mathbb C$.   We denote the associative operation $  \mathcal A \times \mathcal A \to \mathcal A$ by the product notation, that is, we write $ (e,f)\mapsto ef$. In the abstract definition, we do not find useful for $\mathcal A$ itself to be endowed with a topology; on the other hand, we will work with  linear subspaces of $\mathcal A$ endowed with a Banach norm.

We  assume that there exists a subspace $\mathcal L^1 $ of $\mathcal A$ and a linear functional $\tau:\mathcal L^1\to \C$, which we refer to as \emph{trace}.

Given  an $m$-tuple $(X_1,\ldots, X_m)$ of Banach subspaces of $\mathcal A$, we construct the seminorm  
\begin{equation}
\label{e:YJnorm2}
|e|_{  Y(X_1,\ldots, X_m)} = \sup\left\{ \left|\tau\left( e \prod_{\ell=1}^{m} e_{\sigma(\ell)}\right)\right|: \sigma \in\Sigma(m), |e_j|_{  X_j}=1, j =1,\ldots,m \right\}
\end{equation}
on the subspace 
\begin{equation}
\label{e:subspace}
Y(X_1,\ldots, X_m)=
\left\{e\in \mathcal A: e\prod_{\ell=1}^{m} e_{\sigma(\ell)} \in \mathcal L^1 \, \forall \sigma \in \Sigma(m), \, e_j \in X_j,\, j=1,\ldots, m  \right\}
\end{equation}
of $\mathcal A$. 
The next lemma clarifies the intent of definition \eqref{e:YJnorm2}: if $|
\cdot|_{Z}$ is a seminorm  such that all $(m+1)$-linear forms on $X_1\times \cdots \times X_m \times Z$  in \eqref{e:YJnorm3} below are bounded, then the  $Z$-seminorm dominates the seminorm ${Y(}X_1,\ldots, X_m)$.
\begin{lem} \label{lem:max} Let   $(X_1,\ldots, X_m)$ be a $m$-tuple of Banach subspaces of $\mathcal A$. Suppose that $e\in \mathcal A$ belongs to the subspace \eqref{e:subspace}.  Then
\begin{equation}
\label{e:YJnorm3}
\left|\tau\left( e \prod_{\ell=1}^{m} e_{\sigma(\ell)}\right) \right| \leq |e|_Z \prod_{j=1}^m |e_j|_{X_j}, \qquad \forall \sigma \in \Sigma(m), \, e_j \in X_j,\, j=1,\ldots, m,
\end{equation}
holds for $|e|_Z =|e|_{{Y(}X_1,\ldots, X_m)} $. In addition, 
if   $|\cdot|_Z$ is a seminorm on $\mathcal A$ such that 
\eqref{e:YJnorm3} holds, $|e|_{{Y(}X_1,\ldots, X_m)}\lesssim |e|_Z $.
\end{lem}
\begin{proof} Immediate from the definitions.
\end{proof}

\begin{defn}[Admissible spaces]
 \label{def:YJa} We say that a Banach subspace     $X$ of $\mathcal A $ is admissible if  
$Y(X ) 
$ from \eqref{e:subspace} 
  is a Banach space  with respect to    $| \cdot|_{Y(X)}$ of \eqref{e:YJnorm2}\footnote{This includes that if $y\in Y(X)$ then $|y|_{Y(X)}<\infty$.},  
the map 
\begin{equation}
\label{e:mapdual}
y\in Y(X ) \mapsto  x^*[y] \in X^*, \qquad x^*[y](x)=\tau(yx), \quad x\in X,
\end{equation}
is onto, and furthermore, for each $  x\in X, \  y\in Y(X)$, $xy\in \mathcal L^1$ and 
\begin{equation}\label{e:commtrace}
\tau(xy)=\tau(yx).
\end{equation}
\end{defn}
\begin{rem}\label{rem:YJ} 
If $X$ is admissible, then  the map  \eqref{e:mapdual} is an isometric bijection from  $Y(X)$ onto $X^*$. We are thus allowed to identify   
$Y(X)$ with $X^*$ via \eqref{e:mapdual} and we do so without explicit mention  from now on. 
Notice that if $X$ is admissible, then $X$ is a $\UMD$ space if and only if $Y(X)$ is.
\end{rem}

For our purposes, it is convenient to state the next observation in the form of a lemma.

\begin{lem} \label{lem:YJ}
Let $X$ be admissible and reflexive. If $Y(X)$ is also admissible, then $Y(Y(X))=X$ as sets and $|x|_{Y(Y(X))}= |x |_X$ for all $x \in X$.
\end{lem}
\begin{proof} 
The reflexivity of $X$ and Remark \ref{rem:YJ} imply that $Y(Y(X))$ is isometrically isomorphic with $X$. Here we want to show that they are
actually equal as sets with equal norms.
Denote $Y{\coloneqq}Y(X)$ and $Z{\coloneqq} Y(Y)$. It follows quite directly from the definitions that $X$ is a subset of $Z$.

Let $\varphi \colon X^* \to Y$ be the isometric isomorphism from the definition of the admissibility of $X$. This induces the isometric isomorphism
$\phi \colon X^{**} \to Y^*$ defined by 
$$
\phi(x^{**})(y)
{\coloneqq} x^{**}(\varphi^{-1}(y)),
$$
where $x^{**} \in X^{**}$ and $y \in Y$. Since $X$ is reflexive and $Y$ is admissible, we have the canonical isometric isomorphism 
$\rho \colon X \to X^{**}$ and the isometric isomorphism $\eta \colon Y^* \to Z$. Now, the composition $\eta \circ \phi \circ \rho \colon X \to Z$ is an isometric isomorphism.

Suppose $x \in X$ and denote $z{\coloneqq}\eta \circ \phi \circ \rho (x)$.
Let $y \in Y$. Then we have that
$$
\tau(zy)
= \eta^{-1}(z)(y)
=\phi^{-1} \circ \eta^{-1}(z) ( \varphi^{-1}(y))
=\varphi^{-1}(y)(\rho^{-1} \circ \phi^{-1} \circ \eta^{-1}(z))
= \tau(xy).
$$
Since $x$ and $z$ are both elements of $Z$, the fact that $\tau(zy)=\tau(xy)$ for all $y \in Y$ implies that $x=z$. 
Thus, the isometric isomorphism
$\eta \circ \phi \circ \rho \colon X \to Z$ is actually the identity map.
\end{proof}

If $X, X_1, \dots, X_m$ are Banach spaces we write $X=Y(X_1, \dots, X_m)$ to mean that $X$ and $Y(X_1, \dots, X_m)$ coincide as sets,
$Y(X_1, \dots, X_m)$ is a Banach space with the norm $| \cdot |_{Y(X_1, \dots, X_m)}$, and that the norms are equivalent, that is,
$|x|_X \sim |x|_{Y(X_1, \dots, X_m)}$ for all $x \in X$.

We turn to defining  $\UMD$ H\"older $m$-tuples relatively to $\mathcal A$, $\tau$. We first do so for $m=2$.

\begin{defn}[$\UMD$ H\"older pair] \label{defn:productsys0} 
Let $X_1$,  $X_2$ be admissible spaces. We say that $\{X_1,X_2\}$ is a \emph{$\UMD$ H\"older pair}  
if $X_1$ is a $\mathrm{UMD}$ space and $X_2={Y(}X_1)$. 
In view of Remark \ref{rem:YJ} and 
Lemma   \ref{lem:YJ} one can equivalently say that  $\{X_1,X_2\}$ is a $\UMD$ H\"older pair if 
$X_2$ is a $\mathrm{UMD}$ space and  $X_1={Y(}X_2)$.
\end{defn}

 % $\prod_{j\in \calJ_n} e_{\sigma(j)}\in \mathcal L^1,\qquad\left| \tau\left(    \prod_{j\in \calJ_n} e_{\sigma(j)} \right) \right|\leq \prod_{j\in \calJ_n}  | e_{j} |_{X_j}.
%$
For $m\geq 3$ the definition of a $\UMD$ H\"older $m$-tuple is given inductively on $m$ as follows.
\begin{defn}[{$\UMD$ H\"older $m$-tuple, $m\geq 3$}] \label{defn:productsys1} 
Let $X_1,\ldots, X_m$ be admissible spaces. We say that $\{X_1,\ldots, X_{m}\}$ is a \emph{$\UMD$ H\"older $m$-tuple} if
  the following  properties hold.
\vskip2mm \noindent  \textbf{P1.} For all $j_0\in \mathcal J_m$ there holds
\[
X_{j_0}=Y\left(\left\{X_{j}: j\in \mathcal J_{m}\setminus\{j_0\}\right\}\right).
\]
\vskip2mm \noindent \textbf{P2.}
If $1\leq k \leq m-2$ and $\mathcal J=\{j_1<j_2<\cdots <j_k\}\subset \mathcal J_m$, then $Y(X_{j_1}, \dots, X_{j_k})$ is an admissible 
Banach space with the norm \eqref{e:YJnorm2}  and
\begin{equation}\label{eq:SubHolder}
\{X_{j_1}, \dots, X_{j_k}, Y(X_{j_1}, \dots, X_{j_k})\}
\end{equation}
is a $\UMD$ H\"older $(k+1)$-tuple.
\end{defn}

The following remark is an important consequence of the definition. 
\begin{rem}\label{rem:dual} 
Let  $m \ge 3$ and $\{X_1,\ldots, X_{m}\}$   be a $\UMD$ H\"older $m$-tuple.  
Then according to P2 the pair $\{X_{j_0},{Y(}X_{j_0})\}$ is a  $\UMD$ H\"older pair, which by Definition \ref{defn:productsys0}
implies that $X_{j_0}$ and $Y(X_{j_0})$ are $\UMD$ spaces. The inductive nature of the definition then ensures that  
each $Y(X_{j_1}, \dots, X_{j_k})$ appearing in \eqref{eq:SubHolder}  is a $\UMD$ space.
\end{rem}

\begin{rem}\label{rem:ProdInL1}
Let $m \ge 2$ and   $\{X_1, \dots, X_m\}$ be  a $\UMD$ $m$-H\"older tuple.  
Let $e_j \in X_j$ for $j \in \calJ_{m}$. For each $\si \in \Sigma(m)$,
as $X_{\si(1)}=Y(X_{\si(2)}, \dots , X_{\si(m)})$, we necessarily have
$\prod_{j=1}^{m} e_{\si(j)} \in \calL^1$ and 
$$
|\tau(e_{\si(1)} \cdots e_{\si(m)}) |
\le |e_{\si(1)}|_{Y(X_{\si(2)}, \dotsm, X_{\si(m)})}\prod_{j=2}^{m} | e_{\si(j)} |_{X_{\si(j)}}
= \prod_{j=1}^{m} | e_{j} |_{X_{j}}.
$$
\end{rem}

We clarify the extent of our definition with some examples of $\UMD$ H\"older tuples. 
\begin{exmp} It is immediate to verify that the $m$-tuple $X_j=\mathbb C$, $j=1,\ldots, m$, is a $\UMD$ H\"older $m$-tuple with respect to the usual product.
\end{exmp}

The next example is of relevance if one wants to deduce Theorem \ref{thm:X} in the basic case $X_1 = Y_3 = X$ and $X_2 = \C$ from Theorem \ref{thm:scal}.
However, otherwise we do not need it, and Theorem \ref{thm:X} is best seen mimicking our main proofs.
\begin{exmp}\label{ex:tensoralgebra}
\label{ex:Thmx}  Let $X=X_1$ be a complex $\UMD$ space and denote $X_2=X^*$. The goal of this example is to show that for each $m\geq 2$ the tuple $\{
X_1, X_2,\ldots, X_m\}$ with $X_j=\mathbb C$ for $2<j\leq m$ is a $\UMD$ H\"older tuple. This is conceptually simple but requires some work in order to define a suitable enveloping algebra $\mathcal A$. We let $ V=  X\oplus X^*$, and define 
$\mathcal A$ to be the tensor algebra over $V$, namely
\[
\mathcal A = \bigoplus_{k=0}^\infty V^{\otimes k}.
\]
We let \[
\mathcal L^1=\mathrm{span}\{e \otimes e^* + f^* \otimes f,\,  e,f \in X,\,  e^*,f^*\in X^*\};\] notice that this is a linear subspace of $V^{\otimes2}$. We then define the functional     $\tau $    by
\[
\tau \big(e \otimes e^* + f^* \otimes f \big) = \langle f^*,  e \rangle + \langle e^*,  f \rangle
\]
for $e,f\in X$, $e^*,f^*\in X^*$ and extend it to all of $\mathcal L^1$ by linearity.
We notice that the definition \eqref{e:subspace} yields that
\[
Y(X_{j_1}, \dots, X_{j_k})= \begin{cases}
X  & 1 \notin   \{j_1\ldots, j_k\},\, 2   \in  \{j_1\ldots, j_k\},\\
X^* & 1  \in   \{j_1\ldots, j_k\},\, 2 \not\in  \{j_1\ldots, j_k\}, \\
\mathbb C &\{1,2\} \subset \{j_1\ldots, j_k\} \; \textrm{or}\;  \{1,2\}\cap \{j_1\ldots, j_k\}=\varnothing.
\end{cases}
\]
With this information in hand, we learn that $X,X^*, \mathbb C$ are admissible spaces.  Proceeding by induction on $m$, we then easily verify that $\{
X_1, X_2,\ldots, X_m\}$ is a $\UMD$ H\"older tuple.
\end{exmp}

We now start explaining how non-commutative $L^p$ spaces fit our abstract framework.

\begin{exmp}\label{ex:ncex} Consider a von Neumann algebra $\mathcal M\subset \mathcal B(H)$, namely a self-adjoint unital subalgebra of the algebra of bounded linear operators on a complex Hilbert space $H$ which is closed in the weak operator topology \cite{PB,PX}.  Let $\mathcal M_+ =\{A\in \mathcal M: \langle Ah,h\rangle \geq 0 \, \forall h \in H\}$ denote  the positive part of $\mathcal M$. A \emph{trace} $\tau$ is a  functional $\mathcal M_+\to [0, \infty]$ satisfying
\[
\tau(A+\lambda B) = \tau (A) +\lambda \tau(B), \qquad \forall A, B\in \mathcal M_+, \lambda >0
\]
as well as 
 the  tracial property \begin{equation}
\label{e:trace0}
\tau (AA^*) = \tau(A^*A) 
\end{equation}
for all $A  \in \mathcal M$. Following \cite{PX}, we assume $\tau$ is \emph{normal, semifinite, faithful} (n.s.f.) and define    the corresponding space of \emph{measurable operators} $\mathcal A=L^0(\mathcal M)$ equipped with convergence in measure: a detailed definition is in \cite{PX}. Then $\mathcal A$ is a (metrizable) topological  $*$-algebra and $\mathcal M $ is dense in $\mathcal A$. 
We will also  recall the notion of $\mathcal S_+,\mathcal S $ as introduced in    \cite[p.1463]{PX}: $\mathcal S_+$ is the cone of those $A\in \mathcal M_+$ such that $\tau(\mathrm{supp}\, A)<\infty,$ where $\mathrm{supp}\, A$ is the least projection $P\in \mathcal M_+$ with $PA=A$, and $\mathcal S\subset \mathcal M$ is the linear span of $\mathcal S_+$. We note \cite[Proposition 1.15(ii)]{Xu} that $\tau$ may be extended to a unique  linear functional on $\mathcal S$, satisfying 
\begin{equation}
\label{e:trace-1}
\tau(A^*) = \overline{\tau(A)}, \qquad 
\tau(AB)=\tau(BA), \qquad \forall A,B\in \mathcal S.
\end{equation}
For $1\leq p<\infty$, we call \emph{noncommutative $L^p$ space} 
 the Banach subspace of $\mathcal A$ obtained by completion of  $\mathcal S$   with respect to the norm
\[
\|A\|_{L^p(\mathcal M)} = \left[\tau \left(\left(A^\star A  \right)^{\frac p2}\right) \right]^{\frac1p}, \qquad 1\leq p <\infty.
\]
In fact, we record   the characterization
\[
L^p(\mathcal M)=\left\{A\in \mathcal A: \tau \left((A^\star A )  )^{\frac p2}\right)<\infty\right\};
\]
in the above equality, $\tau$ denotes the extension of the trace to the positive part of $\mathcal{A}$ defined via generalized singular numbers \cite{PX}.
We also point out  the H\"older inequality
\begin{equation}
\label{e:holdernc}
\|\xi_1 \xi_2\|_{L^p(\mathcal M)} \leq \|\xi_1  \|_{L^{p_1}(\mathcal M)} \| \xi_2\|_{L^{p_2}(\mathcal M)}, 
\qquad \textstyle \frac{1}{p} =\frac{1}{p_1}+\frac{1}{p_2}
\end{equation}
valid whenever $1\leq p_1,p_2,p<\infty$. A suitable substitute holds for $p=\infty$ if the ${L^p(\mathcal M)} $-norm is replaced by the $\mathcal B(H)$-norm.
 Furthermore, notice that $\tau$ may be extended from $\mathcal S$ to a unique linear bounded functional on $L^1(\mathcal M)$ satisfying
\[
|
\tau(A)| \leq \|A\|_{L^1(\mathcal M)}.
\]
The tracial property \eqref{e:trace-1} extends to the following:  if $A,B \in \calA$ are such that $A\in L^p(\calM)$ and $B\in L^{p'}(\calM)$,  then
\begin{equation}
\label{e:trace}
\tau(AB)=\tau(BA).
\end{equation}
This is the concrete equivalent of property \eqref{e:commtrace} we   assumed in the abstract setup. We refer to \cite[Rem.\ 1.2.11]{Xu} for the details of \eqref{e:trace}. 

For $1<p<\infty$, we then have 
$
L^p(\mathcal M)^* = L^{p'}(\mathcal M)
$
with isometric isomorphism given by the Riesz representation map 
\[
\lambda \in L^p(\mathcal M)^* \mapsto B_\lambda\in  L^{p'}(\mathcal M) , \qquad \lambda(A)= \tau(B_\lambda A) \quad\forall A\in L^p(\mathcal M) .\] 
\emph{A fortiori},   $L^p(\mathcal M)$ is reflexive for $1<p<\infty$. For our purposes, it is also important to observe that $L^p(\mathcal M)$ is a $\UMD$ space in the same range \cite[Corollary 7.7]{PX}.
We detail below two concrete examples of von Neumann algebras equipped with a n.s.f.\ trace. 

  If  $\mathcal M$ is an abelian von Neumann algebra, then $\mathcal M=L^\infty(M,\mu)$ for some measure space $(M,\mathcal \mu)$, a n.s.f. trace is obtained by integration with respect to the measure $\mu$, and $\mathcal A=L^0(M,\mu)$, the topological $*$-algebra of measurable functions on $M$ with respect to convergence in measure. Then  $L^p(\mathcal M)=L^p(M,\mu)$ for $1\leq p <\infty.$

If   $\mathcal M =\mathcal B(H)$, the bounded linear operators over a  separable Hilbert space $H$ and 
\[
\tau(A) = \sum_{j=1}^\infty \langle A e_i, e_i\rangle 
\]
where $e_i$ is any orthonormal basis of $H$ \cite[Example (ii), p.\ 1465]{PX}, then  the spaces $L^p(\mathcal M)$ are   referred to as \emph{Schatten-von Neumann classes} and denoted by $S^p.$

Let now $p_j$, $j=1,\ldots, m$ be a H\"older tuple as in  \eqref{e:holder}. We claim that  $X_j=L^{p_j}(\mathcal M)$ is     a $\UMD$ H\"older tuple relative to the   algebra $\mathcal A=L^0(\mathcal M)$, with trace $\tau$. This can be proved by induction on $m$, relying on the equality
\[
 L^{p(\mathcal J)} ( \mathcal M)= Y(\{L^{p_j}(\mathcal M):j\in \calJ\}), \qquad {  \frac{1}{p(\mathcal J)} }= 1-  \sum_{j \in \mathcal J}  \frac{1}{p_j}
\]
valid for each $\varnothing \subsetneq\mathcal J\subsetneq  \mathcal J_m$, 
whose verification is immediate and left to the reader.
\end{exmp}

\begin{exmp}  \label{ex:thmXj} In  Appendix \ref{a:1}, we prove that if   $\{p_j^{s}:1\leq j\leq m\}$ are H\"older tuples of exponents as in \eqref{e:holder} for $s=0,\ldots,S$, $\mathcal M$ is  a von Neumann algebra  with  n.s.f.\ trace $\tau$ as in Example \ref{ex:ncex}, and $(M_s,\mu_s) $ are    $\sigma$-finite Borel measure spaces for $s=1,\ldots,S$, the tuple of spaces
\begin{equation}
\label{e:exumd}
X_j = L^{p_j^S} (M_S,\mu_S; L^{p_j^{S-1}} (M_{S-1},\mu_{S-1};\cdots L^{p_j^{1}} (M_{1},\mu_{1}; L^{p_j^{0}}(\mathcal M))\cdots)
\end{equation}
is a $\UMD$ H\"older $m$-tuple relative to the   trace
\[
f\mapsto \int\displaylimits_{  M_1 \times\cdots \times M_{S}} \tau (f(t_1,\ldots, t_{S})) \, \mathrm{d} \mu_1 \times\cdots \times \mu_S (t_1,\ldots,t_S).
\]
A precise statement is provided in Proposition \ref{p:a1}.
\end{exmp}

\subsection{Extensions of CZOs}
If $X$ is a Banach space we will use the notation  $L^\infty_c\otimes X$ for functions of the type 
$
\sum_{i=1}^{N} f_{i} e_{i},
$
where $N \in \N$, $f_{i} \in L^\infty_c(\R^d)=:L^\infty_c$ and $e_{i} \in X$.

Let $\{X_1,\ldots, X_{n+1}\}$ be a $ \UMD$ H\"older tuple where $n \ge 1$.
Suppose  $T_0$ is an $n$-linear CZO  with a  kernel $K_0$ as defined in Section \ref{sec:MultiSing}. 
Since we know that $T_0$ is a bounded operator, see \eqref{eq:CZOBdd}, we know that $\langle T_0(f_1, \dots, f_n), f_{n+1} \rangle$
makes sense for $f_j \in L^\infty_c$. 
We define the corresponding $(n+1)$-linear form  
\begin{equation}
\label{e:scalemb}
\begin{split}
&\Lambda_{T_0} :L^\infty_c\otimes X_1 \times \cdots \times L^\infty_c\otimes X_{n+1} \to \C, \\
&\Lambda_{T_0}(f_1,\ldots,f_{n+1})=   \sum_{a_1, \ldots, a_{n+1}}   \langle T_0(f_{1, a_1}, \ldots, f_{n, a_n}), f_{n+1, a_{n+1}}\rangle \tau \Big (  \prod_{j=1}^{n+1} e_{j, a_j} \Big),
\end{split}
\end{equation}
where $f_j=\sum_{a=1}^{N_j} f_{j,a}e_{j,a}$. If $U$ is a dyadic model operator as in Section \ref{sec:MultiSing} 
we define the form $\Lambda_U$ in the corresponding way. 
We can also make sense of $\Lambda_U$ more directly. For example, if $U$ is a dyadic shift as in \eqref{e:shift}, then
\begin{equation}\label{e:shiftforms}
\Lambda_{U}(f_1,\ldots,f_{n+1})
= \sum_{K \in \calD}
\sum_{\substack{Q_1, \dots, Q_{n+1} \in \calD \\ Q_j^{(k_j)}=K}}
a_{K,(Q_j)}\tau\Big(\prod_{j=1}^{n+1} \langle f_j, \wt h_{Q_j} \rangle \Big).
\end{equation}

\begin{rem}
We chose to utilize the identity permutation in $\Sigma(n+1)$ for the product appearing in \eqref{e:scalemb}. However, the notion of being a $\UMD$ H\"older tuple is clearly invariant under reordering of $\{X_1,\ldots, X_{n+1}\}$ . 
\end{rem}

Let $p_j \in (1, \infty)$ for $j \in \calJ_{n+1}$ be such that 
$\sum_{j=1}^{n+1} 1/p_j=1$. From Theorem \ref{thm:scal} it will follow among other things
that
\begin{equation}\label{eq:ForAdjoints}
| \Lambda_{T_0}(f_1, \dots, f_{n+1}) | 
\lesssim \prod_{j=1}^{n+1} \| f_j \|_{L^{p_j}(X_j)}.
\end{equation}
Based on this boundedness one can define as usual $n+1$ adjoint operators.
Let us describe how the adjoints look like in our H\"older tuple set up. 

Fix $j_0\in \mathcal J_{n+1}$ and $f_j\in L^{p_j}( X_j)$ for $j\in \mathcal J_{n+1}\setminus\{j_0\}$. Consider the linear functional
\begin{equation}\label{eq:DefOfAdjoint}
f_{j_0}\in L^{p_{j_0}}( X_{j_0}) \mapsto \Lambda_{T_0}(f_1,\ldots,f_{n+1}),
\end{equation}
which is bounded because of \eqref{eq:ForAdjoints}.
Recall that $L^{p_{j_0}}(X_{j_0})^*$  is identified with  $L^{(p_{j_0})'}(Y(X_{j_0}))$ with duality pairing
\[
\langle g,f_{j_0} \rangle = \int_{\R^d} \tau( g(x) f_{j_0}(x)) \ud x.
\] 
Therefore, there exists a function 
$$
T^{*j_0}(f_j:j \in \calJ_{n+1} \setminus \{j_0\})
{\coloneqq}T^{j_0*}(f_1, \dots, f_{j_0-1}, f_{j_0+1}, \dots, f_{n+1}) \in L^{(p_{j_0})'}(Y(X_{j_0}))
$$ 
so that  
\[
\Lambda_{T_0}(f_1,\ldots,f_{n+1})= \int_{\R^d} \tau( T^{*j_0}(f_j:j \in \calJ_{n+1} \setminus \{j_0\})(x)  f_{j_0}(x)) \ud x.
\]
The $n$-linear bounded operator 
$$
T^{j_0*} \colon L^{p_1}(X_1) \times \dots \times L^{p_{j_0-1}}(X_{j_0-1}) 
\times L^{p_{j_0+1}}(X_{j_0+1}) \times \dots \times L^{p_{n+1}}(X_{n+1})
\to L^{(p_{j_0})'}(Y(X_{j_0}))
$$ 
is one of the adjoint operators.
In the same way one can define the adjoint $T_0^{j_*}$ of $T_0$ so that
$$
\langle T_0^{j_0*}(g_1, \dots, g_{j_0-1}, g_{j_0+1}, \dots, g_{n+1}), g_{j_0} \rangle
= \langle T_0(g_1, \dots, g_n), g_{n+1} \rangle,
$$
where $g_j \in L^{p_j}$.

Suppose $f_j =\sum_{a=1}^{N_j} f_{j,a} e_{j,a} \in L^\infty_c \otimes X_j$ for $j \in \calJ_{n+1} \setminus \{j_0\}$.
A calculation involving the invariance of $\tau$ under cyclic permutations  yields that
\begin{equation}\label{e:scalembadj}
\begin{split}
T^{*j_0}&(f_j:j \in \calJ_{n+1}\setminus\{j_0\})\\
&=\sum 
T_0^{j_0*}(f_{j, a_j}:j \in \calJ_{n+1}\setminus\{j_0\}) e_{j_0+1, a_{j_0+1}} \cdots e_{n+1, a_{n+1}} \cdots e_{j_0-1, a_{j_0-1}}.
\end{split}
\end{equation}

\subsection{Sparse domination of dyadic operators}\label{ss:sparse}
The following basic sparse domination result, Lemma \ref{lem:SparseModel}, was first proved by Culiuc, Ou and one of us in the linear scalar-valued setting in  \cite{CUDPOULMS,CDPOMRL} and recast by Y.\ Ou and three of us in the multilinear scalar-valued case  \cite{LMOV}. The proof in our current Banach-valued setting is completely analogous. 

Let $\eta \in (0,1)$. We say that a collection $\calS$ of cubes in $\R^d$ (not necessarily dyadic) is $\eta$-sparse (or just sparse)
if for every $Q \in \calS$ there exists a set $E_Q \subset Q$ with $|E_Q| > \eta |Q|$ so that the sets $E_Q$, $Q \in \calS$,
are pairwise disjoint.

\begin{lem}\label{lem:SparseModel}
Let $n \geq1$,  $\{X_1,\ldots, X_{n+1}\}$ be a $ \UMD$ H\"older tuple, $\calD$ be a dyadic grid,  $k=(k_1, \dots, k_{n+1})$, $0 \le k_i \in \Z$. Suppose that  the scalars $a_{K,(Q_j)}$ satisfy  the normalization
$$
|a_{K,(Q_j)}| \le A_1 \prod_{j=1}^{n+1} |Q_j|^{1/2}|K|^{-n}
$$
and we are given scalar functions  $u_{j,Q} = \sum_{Q' \in \ch(Q)} c_{j,Q'} 1_{Q'}$ 
satisfying $|u_{j,Q}| \le |Q|^{-1/2}$.

If there exists a H\"older tuple   $p_1, \dots, p_{n+1}$  as in \eqref{e:holder} such that the forms
$$
U_{\calD'}(g_1, \ldots, g_{n+1}) 
{\coloneqq} \sum_{K \in \calD'} \sum_{\substack{Q_1, \dots, Q_{n+1} \in \calD \\ Q_i^{(k_i)}=K}}
a_{K,(Q_i)} \tau\Big( \prod_{j=1}^{n+1} \langle g_j, u_{j,Q_j} \rangle \Big), \qquad \calD' \subset \calD,
$$
satisfy
$$
\sup_{\calD'\subset \calD}
|U_{\calD'}(g_1, \ldots, g_{n+1}) | \le A_2 \prod_{j=1}^{n+1} \|g_j\|_{L^{p_j}(\R^d;X_j)}, \qquad g_j  \in L^\infty_c(\R^d; X_j), j=1,\ldots,n+1,
$$
then for each tuple   $f_j \in L^\infty_c(X_j)$, $j=1,\ldots, n+1$,  and $\eta>0$
there exists an $\eta$-sparse collection $\mathcal{S}=\mathcal{S}((f_j), \eta)\subset \calD$ such that 
\begin{equation}\label{SparseForDMOs}
|\langle U_{\calD}(f_1, \ldots, f_n), f_{n+1}\rangle| 
\lesssim_{\eta} (A_1+A_1\kappa + A_2) \sum_{Q\in\mathcal{S}}|Q|\prod_{j=1}^{n+1}\langle |f_j|_{X_j} \rangle_Q,
\end{equation}
where $\kappa = \max k_m$.
\end{lem}

In the previous lemma the sparse collection is in the same grid where the dyadic operator is defined.
The result can be updated to involve a universal sparse set, which is explained in Remark \ref{rem:universal}.
This is important when we move the sparse estimate from DMOs to CZOs via the representation
theorem, which involves a family of  dyadic grids.

\begin{rem} \label{rem:universal}
There exist dyadic grids $\mathcal{D}_i$, $i = 1, \dots, 3^d$, with the following property, 
see Lacey--Mena \cite{LM},   \cite{LMOV}, or \cite{CDPOBP} for a simple proof.
Let $g_m \in L^1_{\loc}$, $m=1, \dots, n+1$, be scalar-valued and let $\eta_1, \eta_2 \in (0,1)$. 
Then for some $i$ there exists an $\eta_2$-sparse collection $\mathcal{U}=\mathcal{U}((g_m), \eta_2)\subset \calD_i$, 
so that for all $\eta_1$-sparse collections of cubes $\mathcal{S}$ we have
\begin{equation}\label{eq:universalsparse}
\sum_{Q\in\mathcal{S}}|Q|\prod_{m=1}^{n+1}\langle |g_m| \rangle_Q 
\lesssim_{\eta_1, \eta_2} \sum_{Q\in\mathcal{U}}|Q|\prod_{m=1}^{n+1}\langle |g_m| \rangle_Q.
\end{equation}
\end{rem}
\begin{rem}\label{r:sparsemf} In \cite{CDPOBP}, it is noted that the sparse domination estimate for an $n+1$-linear form $\Lambda$ on $\R^d$, acting on scalar functions
\[
|\Lambda(f_1, \ldots, f_{n+1}) | 
\lesssim  \sum_{Q\in\mathcal{S}}|Q|\prod_{j=1}^{n+1}\langle |f_j| \rangle_Q,
\]
is equivalent to the   estimate in terms of the multilinear maximal operator $\mathrm{M}$
\[
|\Lambda(f_1, \ldots, f_{n+1} ) | \lesssim \|\mathrm{M}(f_1,\ldots, f_{n+1})\|_{1}, \qquad \mathrm{M}(f_1,\ldots, f_{n+1})(x) = \sup_{x\in Q } \prod_{j=1}^{n+1}\langle |f_j| \rangle_Q.
\]
Vector-valued versions of this principle may be formulated in a totally analogous way. We have used this equivalence to state   the sparse bounds in our main results; this is particularly convenient as the formulation in terms of the multilinear maximal function may be given without defining what a sparse collection is.
\end{rem}

Next, we discuss the well known fact that the sparse domination of an operator implies boundedness in the full range: for more details and weighted corollaries see \cite{CDPOBP,LMOV} and references therein.

Let $X_1,\ldots, X_{n+1}$ be Banach spaces, $n \ge 1$.
Assume that $\Lambda$ is an $(n+1)$-linear form initially defined on  $L^\infty_c(\R^d)\otimes X_1 \times \cdots \times L^\infty_c(\R^d)\otimes X_{n+1}$ such that 
if $f_j \in L^\infty_c(\R^d)\otimes X_{j}$, then there exists a dyadic lattice $\calD $ 
and a sparse collection $\calS  \subset \calD$
so that
\begin{equation}\label{eq:FormSparse}
| \Lambda(f_1,\ldots,f_{n+1})  |
\lesssim  \sum_{Q\in\mathcal{S}}|Q|\prod_{j=1}^{n+1}\langle |f_j|_{X_j} \rangle_Q.
\end{equation}
This easily implies that if $p_j \in (1, \infty)$ for $j \in \calJ_{n+1}$ are such that $\sum_{j=1}^{n+1} 1/p_j=1$ then
$\Lambda$ can be extended to a bounded form $\Lambda \colon L^{p_1}(X_1) \times \dots \times L^{p_{n+1}}(X_{n+1}) \to \C$.
Indeed, just use H\"older's inequality and then Carleson embedding theorem in the right hand side of \eqref{eq:FormSparse}.

We estimate the adjoints $T^{j*}$ of $\Lambda$, which are defined in the usual way based on the functional as in \eqref{eq:DefOfAdjoint}.
By symmetry it will suffice to tackle the case $j=n+1$ and simply write $T$ in place of $T^{(n+1)*}$. 

We use the so-called $A_\infty$ extrapolation from Cruz-Uribe--Martell--P\'erez \cite{CUMP}. 
Let $A_\infty(\R^d)$ be the class of $A_\infty$ weights in $\R^d$, see \cite{CUMP}
for a definition. Suppose $v \in A_\infty(\R^d)$ and $f_j\in L^\infty_c(X_j)$ for 
$j \in \calJ_n$.  Taking $f_{n+1}(x)=\xi(x)v(x)$ for a suitably chosen $\xi \in L^\infty_c(X_{n+1})$ there holds that
\begin{equation*}
\begin{split} 
\int_{\R^d} |  T(f_1, \dots, f_n) |_{X_{n+1}^*} v \sim \Lambda(f_1,\ldots,f_{n+1}) 
 & \lesssim \sum_{Q \in \calS} \prod_{j=1}^{n}\langle |f_j|_{X_j} \rangle_Q v (Q) \\
&  \le \sum_{Q \in \calS} \Big(\bla M^n_\calD(|f_1|_{X_1}, \dots, |f_n|_{X_n})^{1/2} \bra_Q^v\Big)^2 v(Q) \\
&\lesssim \int_{\R^d}  M^n_\calD(|f_1|_{X_1}, \dots, |f_n|_{X_n})  v,
\end{split}
\end{equation*}
where $\langle h\rangle_Q^v=v(Q)^{-1}\int_Q hv$ and $M^n_\calD(g_1, \dots, g_n){\coloneqq} \sup_{Q \in \calD} \prod_{m=1}^n \langle | g_m | \rangle_Q 1_Q$
is the dyadic maximal function and in the last step we used the Carleson embedding theorem. 
Now, the $A_\infty$ extrapolation result, Theorem 2.1 in \cite{CUMP}, gives that
\begin{equation*}
\int_{\R^d} | T(f_1, \dots, f_n) |_{X_{n+1}^*}^p v
 \lesssim \int_{\R^d}  M^n_\calD(|f_1|_{X_1}, \dots, |f_n|_{X_n})^p  v
\end{equation*}
for all $p \in (0, \infty)$ and $v \in A_\infty(\R^d)$. Using this with $v=1$ the boundedness of the maximal function
gives that
$$
\| T(f_1, \dots, f_n) \|_{L^{q_{n+1}}(X_{n+1}^*)}
\lesssim \prod_{j=1}^n \| f_j \|_{L^{p_j}(X_j)},
$$
where $p_j \in (1, \infty]$ are such that $1/{q_{n+1}}{\coloneqq} \sum_{j=1}^n 1/p_j>0$.
Notice that the boundedness of $M^n_\calD$ follows from H\"older's inequality and the boundedness of $M^1_\calD$,
since  there holds that $M^n_\calD(g_1, \dots, g_n) \le \prod_{m=1}^nM^1_\calD (g_m)$. As is clear, multilinear weighted
bounds also follow from this argument and the corresponding results of $M^n_\calD$.

\subsection{Proof of the main theorem}\label{sec:ProofOfMain}
In this section we state and prove our main theorem assuming the estimates for model operators
from Section \ref{sec:ShiftBdd} and Section \ref{sec:ParaBdd}.

\begin{thm}\label{thm:scal} Let $n\geq 1$, $T_0$ be an $n$-linear CZO with kernel $K_0$ and
$\{X_1,\ldots, X_{n+1}\}$ be a $\UMD$ H\"older tuple.
The $(n+1)$-linear form $\Lambda_{T_0}$ defined in \eqref{e:scalemb} 
can be extended to act on functions $f_j \in L^\infty_c(X_j)$, and
given $\eta \in (0,1)$ there exists an $\eta$-sparse collection of cubes $\mathcal{S}=\mathcal{S}((f_m), \eta)$ so that
\begin{equation}\label{SparseForTsc}
\begin{split}
|\Lambda_{T_0}(f_1, \ldots,  f_{n+1})| \lesssim_{\eta}  
\Big[\| K_0 \|_{\operatorname{CZ}_\alpha} + \|T_0\|_{\WBP  } + \sum_{j=1}^{n+1}& \|(T_0)^{j*}(1, \ldots, 1)\|_{\BMO }\Big] \\
&\times \sum_{Q\in\mathcal{S}}|Q|\prod_{j=1}^{n+1}\langle |f_j|_{X_j} \rangle_Q.
\end{split}
\end{equation}
Consequently, we for instance have
$$
\| T_0(f_1, \dots, f_n) \|_{L^{q_{n+1}}(X_{n+1}^*)}
\lesssim \prod_{j=1}^n \| f_j \|_{L^{p_j}(X_j)}
$$
whenever $p_j \in (1, \infty]$ are such that $1/{q_{n+1}}{\coloneqq} \sum_{j=1}^n 1/p_j>0$. See Section \ref{ss:sparse} for a full discussion of the corollaries of the sparse estimate.
\end{thm}

\begin{proof}
Let $f_j \in L^\infty_c \otimes X_j$ for $j \in \calJ_{n+1}$ be of the form $f_j =\sum_{a=1}^{N_j} f_{j,a} e_{j,a}$.
Then, we have by the dyadic representation \eqref{eq:repthm} that
\begin{equation}\label{eq:ApplyRep}
\begin{split}
 &\Lambda_{T_0}( f_1,\ldots, f_{n+1}) \\
&= C_T \sum_{a_1, \ldots, a_{n+1}}  \E_\omega \sum_{k_1, \ldots, k_{n+1} =0}^\infty  \sum_u 2^{-\frac{\alpha\max k_i }{2}}
\langle U^k_{\calD_\omega,u}(f_{1, a_1}, \ldots, f_{n, a_n}), f_{n+1, a_{n+1}}\rangle \tau\Big( \prod_{j=1}^{n+1} e_{j, a_j}\Big) \\
&= C_T \E_\omega \sum_{k_1, \ldots, k_{n+1} =0}^\infty  \sum_u 2^{-\frac{\alpha\max k_i }{2}} 
\Lambda_{U^k_{\calD_\omega,u}} (f_1,\ldots,f_{n+1}).
\end{split}
\end{equation}

In Section \ref{sec:ShiftBdd} and Section \ref{sec:ParaBdd} it is shown that if $U$ is a dyadic model operator then
\begin{equation}\label{eq:ModelsBdd}
| \Lambda_U(g_1, \dots, g_{n+1})|
\lesssim \prod_{j=1}^{n+1} \| g_j \|_{L^{p_j}(X_j)}
\end{equation}
holds for any $p_j \in (1, \infty)$ and $g_j \in L^\infty_c(X_j)$, $j \in \calJ_{n+1}$, such that $\sum_{j=1}^{n+1} 1/p_j=1$; if $U$ is a shift, then the estimate
depends polynomially on the complexity.
This implies that $\Lambda_{T_0}$ can be extended to act on functions $f_j \in L^\infty_c(X_j)$ and that
\eqref{eq:ApplyRep} holds for such functions.

The estimate \eqref{eq:ModelsBdd} implies via Lemma \ref{lem:SparseModel} 
and Remark \ref{rem:universal} that if $f_j \in L^\infty_c(X_j)$ for $j \in \calJ_{n+1}$
then there exist a dyadic grid and
an $\eta$-sparse collection $\calS=\calS((f_j))\subset \calD$ so that all the model operators appearing in \eqref{eq:ApplyRep} satisfy
$$
| \Lambda_{U^k_{\calD_\omega,u}} (f_1,\ldots,f_{n+1})|
\lesssim \sum_{Q\in\mathcal{S}}|Q|\prod_{j=1}^{n+1}\langle |f_j|_{X_j} \rangle_Q,
$$
where the estimate depends polynomially on the complexity.
This combined with \eqref{eq:ApplyRep} finishes the proof.
\end{proof}
In Section \ref{s:maxhtup},   we show that   the $\UMD$ H\"older tuples enjoy a suitable maximal property  among tuples of spaces admitting $L^p$-bounded    extensions of $n$-linear CZO operators and dyadic shifts.

\section{Boundedness of multilinear shifts in $\UMD$ H\"older tuples}\label{sec:ShiftBdd}

This section is dedicated to the proof of the boundedness of multilinear shifts.
Before starting the main argument, we record a randomized bound for $\UMD$ H\"older tuples in the following lemma.

\begin{lem} \label{l:Rscalar}
Let $\{X_1,\ldots, X_{n+1}\}$ be a $\UMD$ H\"older tuple, $n \ge 2$, and let $K \in \Z_+$. For each $k=1,\ldots, K$
let $a_k$ be a scalar such that $|a_k| \le 1$ and for each $j \in \calJ_{n}$ assume that we are given $e_{j, k} \in X_j$.
Then we have
$$
\Big | \sum_{k=1}^K  a_k \prod_{j=1}^n e_{j, k} \Big|_{{Y(}X_{n+1})} \le \prod_{j=1}^n \|(e_{j,k})_{k=1}^K\|_{\Rad(X_j)}.
$$
\end{lem}

\begin{proof}
Fix $K$, $|a_k| \le 1$ and $e_{j, k} \in X_j$ as in the assumptions.
Let $\{\varepsilon^i_k\}_{k=1}^K$, $i \in \calJ_{n-1}$, be collections of  independent random signs.
We denote the expectation with respect to the random variables $\{\varepsilon^i_k\}_{k=1}^K$ by $\E^i$, and write $\E=\E^1 \cdots \E^{n-1}$.
We have the identity
\begin{equation*}
\begin{split}
\sum_{k=1}^K  a_k \prod_{j=1}^n e_{j, k} 
&= \E \sum_{k_1, \dots, k_{n} =1}^K \varepsilon^1_{k_1}\varepsilon^1_{k_2}
\varepsilon^2_{k_2}\varepsilon^2_{k_3} \cdots \varepsilon_{k_{n-1}}^{n-1}\varepsilon_{k_{n}}^{n-1}
a_{k_1} \prod_{j=1}^n e_{j, k_j}\\
& =\E  \Big(  \sum_{k_1=1}^K \varepsilon^1_{k_1} a_{k_1} e_{1,k_1} \Big)
\Big(  \sum_{k_2 =1}^K \varepsilon^1_{k_2}\varepsilon^2_{k_2} e_{2,k_2} \Big) \cdots 
\Big(  \sum_{k_n=1}^K \varepsilon_{k_{n}}^{n-1}e_{n, k_n} \Big).
\end{split}
\end{equation*}
We can dominate this with
$$
\E   \Big  \| \sum_{k_1 =1}^K \varepsilon^1_{k_1} a_{k_1} e_{1,k_1} \Big\|_{X_1}
\Big(\prod_{i=2}^{n-1} \Big \| \sum_{k_i =1}^K \varepsilon^{i-1}_{k_i}\varepsilon^i_{k_i} e_{i,k_i} \Big \|_{X_i} \Big)
\Big  \| \sum_{k_n=1}^K \varepsilon_{k_{n}}^{n-1}e_{n, k_n} \Big\|_{X_n},
$$
which is further controlled by
\begin{equation} \label{e:radlemma}
\begin{split}
&\quad \Big(\E  \Big  \| \sum_{k_1 =1}^K \varepsilon^1_{k_1} a_{k_1} e_{1,k_1} \Big\|_{X_1}^2\Big)^{1/2} \\
&\times
\Big[\E\Big(\prod_{i=2}^{n-1} \Big \| \sum_{k_i =1}^K \varepsilon^{i-1}_{k_i}\varepsilon^i_{k_i} e_{i,k_i} \Big \|_{X_i}^2 
\Big)\Big  \| \sum_{k_n=1}^K \varepsilon_{k_{n}}^{n-1}e_{n, k_n} \Big\|_{X_n}^2
 \Big]^{1/2}.
\end{split}
\end{equation}

The first factor is less than $\| (e_{1,k})_{k=1}^K \| _{\Rad(X_1)}$ by Kahane's contraction principle.
We now consider the second factor. We see that the variables $\varepsilon^1_{k}$ appear only inside the norm $X_{2}$, and moreover
there holds that
$$
\E^1  \Big \| \sum_{k_2=1}^K \varepsilon^1_{k_2}\varepsilon^2_{k_2} e_{2,k_2} \Big \|_{X_2}^2 
= \| (e_{2,k})_{k=1}^K \| _{\Rad(X_2)}^2.
$$
After using this identity, the variables $\varepsilon_k^1$ do not appear anymore, and the
variables $\varepsilon^2_k$ appear only inside the norm  $X_3$. Repeating the same reasoning, we deduce that the second factor in \eqref{e:radlemma} is equal to the product $\prod_{j=2}^{n} \| (e_{j,k})_{k=1}^K \| _{\Rad(X_j)}$.
\end{proof}

Now, we turn to the actual proof of boundedness of shifts.
We assume that $n \ge 1$ and that $\{X_1, \dots, X_{n+1}\}$ is a $\UMD$ H\"older tuple.
Let $k=(k_1, \dots, k_{n+1})$, $0 \le k_i \in \Z$,
 and let $\calD$ be a dyadic lattice in $\R^d$.
Suppose $S^k{\coloneqq}S^k_\calD$ is an $n$-linear dyadic shift as described in Equation \eqref{e:shift}. 
We consider its related $(n+1)$-linear form $\Lambda_{S^k}$ which acts on
locally integrable functions $f_j \colon \R^d \to X_j$ by
\begin{equation}\label{eq:DyadicShiftForm}
\Lambda_{S^k}(f_1,\dots,f_{n+1})
=\sum_{K \in \calD} \Lambda_K(f_1, \ldots, f_{n+1}),
\end{equation}
where 
$$
\Lambda_K(f_1, \ldots, f_{n+1}) = 
\sum_{\substack{Q_1, \dots, Q_{n+1} \in \calD \\ Q_j^{(k_j)}=K}}
a_{K,(Q_j)_{j \in \calJ_{n+1}}} \tau \Big(\prod_{j=1}^{n+1}  \langle f_j, \wt h_{Q_j} \rangle\Big).
$$
Here $a_{K,(Q_j)_{j \in \calJ_{n+1}}}$ is a scalar satisfying $|a_{K,(Q_j)_{j \in \calJ_{n+1}}}| \le \prod_{j=1}^{n+1} |Q_j|^{1/2} |K|^{-n}$,
and there exist two indices $j_0,j_1 \in \calJ_{n+1}$, $j_0 \not =j_1$, so that $\wt h_{Q_{j_0}}=h_{Q_{j_0}}$, $\wt h_{Q_{j_1}}=h_{Q_{j_1}}$ and $\wt h_{Q_j}=h_{Q_j}^0$ if $j \in \calJ_{n+1} \setminus \{j_0, j_1\}$.

 The sparse domination lemma \ref{lem:SparseModel} reduces the problem to  the following theorem.

\begin{thm}\label{thm:MultiShifts}
Suppose $p_j \in (1, \infty)$ for $j \in \calJ_{n+1}$ are such that 
$\sum_{j=1}^{n+1}1/p_j=1$.
The dyadic shift form from \eqref{eq:DyadicShiftForm} satisfies the estimate
$$
|\Lambda_{S^k}(f_1,\dots,f_{n+1})| \lesssim \prod_{j=1}^{n+1} \| f_j \|_{L^{p_j}(X_j)}
$$
for $f_j \in L^\infty_c(X_j)$, where the estimate depends polynomially on $\kappa{\coloneqq} \max_jk_j$.
\end{thm}

\begin{proof}
Let $f_j \in L^\infty_{c}(X_j)$ for $j \in \calJ_{n}$ and consider \eqref{eq:DyadicShiftForm}.
Recall the lattices $\calD_{i,\kappa}$ from  \eqref{eq:SubLattice}, where $\kappa{\coloneqq}\max_j k_j$.
We first divide the sum over the cubes $K \in \calD$ as $\sum_{i=0}^{\kappa}\sum_{K \in \calD_{i,\kappa}}$. 
We fix one $i$ and consider the corresponding term.

Let $ \wt \calJ$ be the set of those indices such that the corresponding Haar functions are non-cancellative, that is, 
$\wt h_{Q_j}=h^0_{Q_j}$. Suppose $j \in \wt \calJ$ is such that $k_j>0$. We use that fact that 
$\langle f_j, \wt h_{Q_j} \rangle=\langle E_K^{k_j} f_j,  h^0_{Q_j} \rangle$ and split
\begin{equation}\label{eq:NonCancExpand}
E^{k_j}_K f_j= \sum_{l_j=0}^{k_j-1} \Delta^{l_j}_K f_j + E_K f_j.
\end{equation}
There holds that
\begin{equation}\label{eq:Id1}
\langle E_K f_j, h^0_{Q_j} \rangle 
= \langle f_j, h^0_K \rangle \langle h_K^0, h_{Q_j}^0\rangle
\end{equation}
and
\begin{equation}\label{eq:Id2}
\langle \Delta^{l_j}_K f_j, h^0_{Q_j} \rangle 
= \langle f_j, h_{Q_j^{(k_j-l_j)}} \rangle \langle h_{Q_j^{(k_j-l_j)}}, h^0_{Q_j} \rangle,
\end{equation}
where as usual we suppressed the summation over the different Haar functions.

We use \eqref{eq:NonCancExpand} to split $ \sum_{K \in \calD_{i,\kappa}} \Lambda_K(f_1, \dots, f_{n+1})$ into at most $(1+\kappa)^{n-1}$ terms of the form
\begin{equation}\label{eq:SplitShift}
\sum_{K \in  \calD_{i,\kappa}} \sum_{\substack{Q_1, \dots, Q_{n+1} \in \calD \\ Q_j^{(k_j)}=K}}
a_{K,(Q_j)_{j \in \calJ_{n+1}}} \tau \Big(\prod_{j=1}^{n+1} \langle P^{l_j}_{K,j} f_j, \wt h_{Q_j} \rangle \Big),
\end{equation}
where $l_j \in \Z$, $0 \le l_j \le k_j$. 
For $j \in \calJ_{n+1} \setminus \wt \calJ$ we have that $P^{l_j}_{K,j}$ is the identity operator, and below we write $l_j=k_j$.
If $j \in \wt  \calJ$  and $l_j>0$ then $P^{l_j}_{K,j}=\Delta^{l_j}_{K}$, and if $j \in \wt \calJ$ and $l_j=0$ then $P^{0}_{K,j}$ is either $E_K$ or $\Delta_K$ 
(but does not change with $K$). We write
$$
\sum_{K \in  \calD_{i,\kappa}} \sum_{\substack{Q_1, \dots, Q_{n+1} \in \calD \\ Q_j^{(k_j)}=K}}
= \sum_{K \in  \calD_{i,\kappa}}  \sum_{\substack{L_1, \dots, L_{n+1} \in \calD \\ L_j^{(l_j)}=K}}  \sum_{\substack{Q_1, \dots, Q_{n+1} \in \calD \\ Q_j^{(k_j-l_j)}=L_j}}
$$
and notice that by \eqref{eq:Id1} and \eqref{eq:Id2} we always have that
$$
\langle P^{l_j}_{K,j} f_j, \wt h_{Q_j} \rangle = \langle f_j, h_{L_j}'\rangle \gamma(Q_j, L_j),
$$
where $h_{L_j}' \in \{h_{L_j}^0, h_{L_j}\}$ and
$$
|\gamma(Q_j, L_j)| = \frac{|Q_j|^{1/2}}{|L_j|^{1/2}}.
$$
We can now write \eqref{eq:SplitShift} further as
\begin{equation}\label{eq:NewShiftTypeOper}
\sum_{K \in  \calD_{i,\kappa}} \sum_{\substack{L_1, \dots, L_{n+1} \in \calD \\ L_j^{(l_j)}=K}} 
b_{K,(L_j)_{j \in \calJ_{n+1}}} \tau \Big(\prod_{j=1}^{n+1} \langle f_j,  h_{L_j}' \rangle \Big),
\end{equation}
where
$$
b_{K,(L_j)_{j \in \calJ_{n+1}}}
=\sum_{\substack{Q_1, \dots, Q_{n+1} \in \calD \\ Q_j^{(k_j-l_j)}=L_j}}
 a_{K,(Q_j)_{j \in \calJ_{n+1}}} \prod_{j=1}^{n+1} \gamma(Q_j, L_j).
$$
There exists $\calJ \subset \calJ_{n+1}$ with $\#\calJ \ge 2$
so that $h'_{L_j}=h_{L_j}$ for $j \in \calJ$ and if $j \in \calJ_{n+1} \setminus \calJ$ then $h'_{L_j}=h^0_{L_j}$ and $l_j=0$. 
Also, we have the normalization
$
|b_{K,(L_j)_{j \in \calJ_{n+1}}}| \le \prod_{j=1}^{n+1} |L_j|^{1/2} |K|^{-n}.
$

We have reduced to considering the new shift type operator \eqref{eq:NewShiftTypeOper}. 
The coefficients satisfy the usual normalization of shifts, but the number $\#\calJ$ of indices with cancellative Haar functions may be bigger than $2$.
What is essential is that the complexity related to the non-cancellative indices is zero -- that is, if $j \in \calJ_{n+1} \setminus \calJ$ then $l_j=0$.
We now start estimating \eqref{eq:NewShiftTypeOper}. Also, the separation of scales, $K \in \calD_{i,\kappa}$, will allow us to use the decoupling estimate \eqref{eq:DecEst}.

\textbf{Case 1.} Assume that $\calJ = \calJ_{n+1}$. Let $q_{n+1} \in (1, \infty)$ be the exponent determined by $1/q_{n+1}=\sum_{j=1}^n 1/p_j$.
We need to estimate
\begin{equation*}
\begin{split}
&\Big\|\sum_{K \in  \calD_{i,\kappa}} \sum_{\substack{L_1, \dots, L_{n+1} \in \calD \\ L_j^{(l_j)}=K}} 
b_{K,(L_j)_{j \in \calJ_{n+1}}} \prod_{j=1}^{n} \langle f_j,  h_{L_j} \rangle h_{L_{n+1}} \Big\|_{L^{q_{n+1}}( {Y(}X_{n+1}))} \\
&\sim \Big( \E \int_{\R^d} \int_{\calV} 
\Big | \sum_{K \in \calD_{i,\kappa}} \varepsilon_K 1_K(x) \sum_{\substack{L_1, \dots, L_{n+1} \in \calD \\ L_j^{(l_j)}=K}} 
b_{K,(L_j)_{j \in \calJ_{n+1}}} \\
&\hspace{6cm}\times \prod_{j=1}^{n} \langle f_j,  h_{L_j} \rangle h_{L_{n+1}}(y_K) \Big|^{q_{n+1}}_{{Y(}X_{n+1})}
\ud \nu (y) \ud x \Big)^{1/q_{n+1}},
\end{split}
\end{equation*}
where we used the decoupling estimate. Notice that since by assumption $X_{n+1}=Y(X_1, \dots, X_n)$, 
there holds also that $Y(X_{n+1})=Y(Y(X_1, \dots, X_n))$, so we could also use the norm $| \cdot |_{Y(Y(X_1, \dots, X_n))}$ instead.
Write
\begin{align*}
&\sum_{\substack{L_1, \dots, L_{n+1} \in \calD \\ L_j^{(l_j)}=K}} 
b_{K,(L_j)_{j \in \calJ_{n+1}}} \prod_{j=1}^{n} \langle f_j,  h_{L_j} \rangle h_{L_{n+1}}(y_K) \\
&= \frac{1}{|K|^n} \int_{K^n} b_K(y_K, z) \prod_{j=1}^n \Delta_{K}^{l_j} f_j(z_j) \ud z 
= \int_{\calV^n} b_K(y_K, z_K)\prod_{j=1}^n \Delta_{K}^{l_j} f_j(z_{j,K}) \ud \nu_n(z),
\end{align*}
where $\nu_n$ is the product measure $\nu \times \cdots \times \nu$ on the product space $\calV^n$ and
$$
b_K(y_K, z_K) = |K|^n\sum_{\substack{L_1, \dots, L_{n+1} \in \calD \\ L_j^{(l_j)}=K}}  
b_{K,(L_j)_{j \in \calJ_{n+1}}} \prod_{j=1}^n h_{L_j}(z_{j,K})  h_{L_{n+1}}(y_K).
$$

We can now continue the estimate by using H\"older's inequality related to the integral $\int_{\calV^n}$. We end up with
\begin{equation}\label{eq:ReadyForR}
\Big( \E \int_{\R^d} \int_{\calV} \int_{\calV^n}
\Big | \sum_{K \in \calD_{i,\kappa}} \varepsilon_K 1_K(x)  
  b_K(y_K, z_K)\prod_{j=1}^n \Delta_{K}^{l_j} f_j(z_{j,K}) \Big|^{q_{n+1}}_{{Y(}X_{n+1})}
\ud \nu_n(z) \ud \nu (y) \ud x \Big)^{1/q_{n+1}}.
\end{equation}
Suppose $n \ge 2$. Notice that $|b_K(y_K, z_K)| \le 1$ and use Lemma \ref{l:Rscalar} to get that
\begin{align*}
\Big| \sum_{K \in \calD_{i,\kappa}} \varepsilon_K 1_K(x)  
  b_K(y_K, z_K)\prod_{j=1}^n \Delta_{K}^{l_j} f_j(z_{j,K}) \Big|_{{Y(}X_{n+1})} \\
\le  \prod_{j =1}^n \| (1_K(x)  \Delta^{l_j}_Kf_j(z_{j,K}))_{K \in \calD_{i, \kappa}} \|_{\Rad(X_j)}.
\end{align*}
Using first H\"older's inequality, then Kahane-Khintchine inequality and finally the decoupling estimate, we conclude that
\begin{align*}
\eqref{eq:ReadyForR}
&\lesssim \prod_{j=1}^n \Big( \int_{\R^d} \int_{\calV} \| (1_K(x)  \Delta^{l_j}_Kf_j(z_{K}))_{K \in \calD_{i, \kappa}} \|_{\Rad(X_j)}^{p_j} \ud \nu(z) \ud x\Big)^{1/p_j} \\
&\sim \prod_{j=1}^n \Big( \E  \int_{\R^d} \int_{\calV} \Big|\sum_{K \in \calD_{i, \kappa}} \varepsilon_K 1_K(x) \Delta^{l_j}_Kf_j(z_{K}) \Big|_{X_j}^{p_j} \ud \nu(z) \ud x\Big)^{1/p_j} \\
&\lesssim \prod_{j=1}^n \| f_j \|_{L^{p_j}(X_j)}.
\end{align*}

Suppose then $n=1$. In this case we have that $q_2=p_1$ and $Y(X_2)=X_1$. 
We use Kahane-Khintchine inequality to move the expectation inside of the exponent $p_1$. 
Then, we use Kahane's contraction principle and move the expectation out again.  This gives that
$$
\eqref{eq:ReadyForR}
\lesssim \Big( \E \int_{\R^d} \int_{\calV}
\Big | \sum_{K \in \calD_{i,\kappa}} \varepsilon_K 1_K(x)  
 \Delta_{K}^{l_1} f_1(z_{K}) \Big|^{p_{1}}_{X_1}
\ud \nu(z)  \ud x \Big)^{1/p_1}
\lesssim \| f_1 \|_{L^{p_1}(X_1)},
$$
where the last step used the decoupling estimate. Linear estimates for shifts have appeared e.g. in \cite{HH, Hy2}.

\textbf{Case 2.} Assume now that $\calJ \subsetneq \calJ_{n+1}$. Since $\#\calJ \ge 2$, this implies that $n \ge 2$.  
Let $j_0 \in \calJ_{n+1} \setminus \calJ$  be an index such that $j_0+1 \in \calJ$; by $(n+1)+1$ we mean $1$. 
Let $\si \in \Sigma(n+1)$ be the cyclic permutation such that $\si(n)=j_0$. Then $\si(n+1) \in \calJ$.
If $e_j \in X_j$ for $j \in \calJ_{n+1}$ then from Remark \ref{rem:ProdInL1} one sees that $\prod_{j=1}^n e_j \in Y(X_{n+1})$
and therefore the cyclic invariance of the trace \eqref{e:commtrace} gives that
$
\tau(e_1 \cdots e_{n+1})
=\tau(e_{n+1} e_1 \cdots e_n).
$    
Repeating this we have that
\eqref{eq:NewShiftTypeOper} is equal to
$$
\sum_{K \in  \calD_{i,\kappa}} \sum_{\substack{L_1, \dots, L_{n+1} \in \calD \\ L_j^{(l_j)}=K}} 
b_{K,(L_j)_{j \in \calJ_{n+1}}} \tau \Big(\prod_{j=1}^{n+1} \langle f_{\si(j)},  h_{L_{\si(j)}}' \rangle \Big).
$$
Having made this important observation, we may now assume, for small notational convenience, that $j_0 = n$ and $\sigma = \id$.
Under this assumption $n \in \calJ_{n+1} \setminus \calJ$, which implies that $l_n=0$. Thus, the coefficient $b_{K,(L_j)_{j \in \calJ_{n+1}}}$
depends only on the cubes $L_1, \dots, L_{n-1}, L_{n+1}$ and $K$. Below we will write the coefficient as $b_{K,(L_j)}$.

We need to estimate
\begin{equation*}
\begin{split}
\Big\|\sum_{K \in  \calD_{i,\kappa}}& \sum_{\substack{L_1, \dots, L_{n-1}, L_{n+1} \in \calD \\ L_j^{(l_j)}=K}} 
b_{K,(L_j)} \prod_{j=1}^{n-1} \langle f_j,  h_{L_j}' \rangle \langle f_{n} \rangle_K |K|^{1/2} h_{L_{n+1}} \Big\|_{L^{q_{n+1}}( {Y(}X_{n+1}))} \\
&\sim \Big( \E \int_{\R^d} \int_{\calV} 
\Big | \sum_{K \in \calD_{i,\kappa}} \varepsilon_K 1_K(x) \langle \varphi_{K,y} \rangle_K  \Big|^{q_{n+1}}_{{Y(}X_{n+1})}
\ud \nu (y) \ud x \Big)^{1/q_{n+1}},
\end{split}
\end{equation*}
where we used the decoupling estimate, and for $K \in \calD_{i,\kappa}$ and $y \in \calV$ we defined the function $\varphi_{K,y} \colon \R^d \to {Y(}X_{n+1})$ by setting $\varphi_{K,y}(x)$ to equal
\begin{equation*}
\begin{split}
|K|^{1/2} \sum_{\substack{L_1, \dots, L_{n-1}, L_{n+1} \in \calD \\ L_j^{(l_j)}=K}}  
b_{K,(L_j)} \prod_{j=1}^{n-1} \langle f_j,  h_{L_j}' \rangle  f_{n}(x)  h_{L_{n+1}}(y_K).
\end{split}
\end{equation*}
After using Stein's inequality \eqref{eq:SteinUMD} with respect to $x \in \R^d$ with fixed $y \in \calV$ we are left with
\begin{equation}\label{eq:ShiftAfterStein1}
\Big( \E \int_{\R^d} \int_{\calV} 
\Big | \sum_{K \in \calD_{i,\kappa}} \varepsilon_K 1_K(x) \varphi_{K,y}(x)  \Big|^{q_{n+1}}_{{Y(}X_{n+1})}
\ud \nu (y) \ud x \Big)^{1/q_{n+1}}.
\end{equation}

Recall that $n \ge 2$ in Case 2. 
From Remark \ref{rem:ProdInL1} we can deduce that if $e_n \in X_n$ and $e_{n+1} \in X_{n+1}$, then $e_ne_{n+1} \in Y(X_1, \dots, X_{n-1})$ and 
$|e_ne_{n+1}|_{Y(X_1, \dots, X_{n-1})} \le |e_n|_{X_n} | e_{n+1} |_{X_{n+1}}$.
Also, since $\{X_1, \dots, X_{n-1}, Y(X_1, \dots, X_{n-1})\}$ is a $\UMD$ H\"older tuple, 
we see from Remark \ref{rem:ProdInL1}  again that if $e_j \in X_j$ for $j \in \calJ_{n-1}$, then $\prod_{j=1}^{n-1} e_j \in Y(Y(X_1, \dots, X_{n-1}))$.
Suppose now that $e_{j,k}\in X_j$ for $j \in  \calJ_{n-1}$, $k=1, \dots, K$, and $e_{n} \in X_{n}$. Then the above consideration implies that 
the key inequality
\begin{equation}\label{eq:KeyIneq}
\Big | \sum_{k=1}^K \prod_{j=1}^{n-1} e_{j,k}e_{n} \Big |_{{Y(}X_{n+1})}
\le \Big | \sum_{k=1}^K \prod_{j=1}^{n-1} e_{j,k} \Big |_{Y(Y(X_1, \dots, X_{n-1}))} |e_{n}|_{X_{n}}
\end{equation}
holds. Write $Z{\coloneqq}Y(Y(X_1, \dots, X_{n-1}))$ for the moment.
Using this in \eqref{eq:ShiftAfterStein1} and then H\"older's inequality 
we have that \eqref{eq:ShiftAfterStein1} is dominated by $\| f_n \|_{L^{p_n}(X_n)}$ multiplied by
\begin{equation*}
\begin{split}
\Big( \E \int_{\R^d} & \int_{\calV} 
\Big |\sum_{K \in \calD_{i,\kappa}} \varepsilon_K 1_K(x)  \sum_{\substack{L_1, \dots, L_{n-1}, L_{n+1} \in \calD \\ L_j^{(l_j)}=K}}  
\wt b_{K,(L_j)} \prod_{j=1}^{n-1} \langle f_j,  h_{L_j}' \rangle   h_{L_{n+1}}(y_K)  \Big|^{p(\calJ_{n-1})}_{Z}
\ud \nu (y) \ud x \Big)^{\frac{1}{p(\calJ_{n-1})}}  \\
& \sim \Big \|\sum_{K \in \calD_{i,\kappa}} \sum_{\substack{L_1, \dots, L_{n-1}, L_{n+1} \in \calD \\ L_j^{(l_j)}=K}}  
\wt b_{K,(L_j)} \prod_{j=1}^{n-1} \langle f_j,  h_{L_j}' \rangle   h_{L_{n+1}} \Big \|_{L^{p(\calJ_{n-1})}(Z)},
\end{split}
\end{equation*}
where we defined $1/p(\calJ_{n-1}){\coloneqq} \sum_{j=1}^{n-1} 1/p_j$, $\wt b_{K,(L_j)}{\coloneqq} |K|^{1/2} b_{K,(L_j)}$ and used the decoupling inequality. 
Notice that 
$$
|\wt b_{K,(L_j)}| \le \frac{ \prod_{j=1}^{n-1} |L_j|^{1/2} |L_{n+1}|^{1/2}}{ |K|^{n-1}}.
$$

We see that we have reduced the estimate to the boundedness of an $(n-1)$-linear shift type operator as in \eqref{eq:NewShiftTypeOper}.
Now, we have two possibilities. If all the Haar functions $h'_{L_j}$ for $j \in \calJ_{n-1}$ are cancellative then we are in a position to apply Case 1 from above
to finish the estimate. 
If some of them is non-cancellative, then we dualize with a function $g \in L^{p(\calJ_{n-1})'}(Y(X_1, \dots, X_{n-1}))$.
This leads us to a corresponding situation as the beginning of Case 2 above but now the form is $n$-linear and  
the underlying $\UMD$ H\"older $n$-tuple is $\{X_1, \dots, X_{n-1}, Y(X_1, \dots, X_{n-1})\}$.
We see that we can repeat the argument in Case 2 until we can apply Case 1. This finishes the proof.

\end{proof}

\begin{rem}\label{rem:2VSn}
We discuss why Theorem \ref{thm:X} works without any $\UMD$ H\"older tuple assumptions on the spaces $X_1$, $X_2$ and $Y_3$, and why we can't allow more 
$\UMD$ spaces in Theorem \ref{thm:X}.
The key point is that for $e_{1,k} \in X_1$ and $e_2 \in X_2$ the estimate
\begin{equation}\label{eq:BilTrivEst}
\Big | \sum_{k=1}^K  e_{1,k}e_{2} \Big |_{Y_3}
\le \Big | \sum_{k=1}^K e_{1,k} \Big |_{X_1} |e_{2}|_{X_{2}},
\end{equation}
which corresponds to \eqref{eq:KeyIneq}, holds without any further assumptions on the spaces.
Using this kind of estimates one can prove Theorem \ref{thm:X} with similar techniques as in the proof of Theorem \ref{thm:MultiShifts}.

Suppose then we have $\UMD$ spaces $X_1, \dots, X_n$ and $Y_{n+1}$, where $n \ge 3$, and we have a product 
$X_1 \times \dots \times X_n \to Y_{n+1}$ -- a bounded $n$-linear operator.
 Of course, an estimate
corresponding to \eqref{eq:BilTrivEst} holds, namely
\begin{equation}\label{eq:nLinTrivEst}
\Big | \sum_{k=1}^K  e_{1,k} \prod_{j=2}^ne_{j} \Big |_{Y_{n+1}}
\le \Big | \sum_{k=1}^K e_{1,k} \Big |_{X_1} \prod_{j=2}^n|e_{j}|_{X_{j}}.
\end{equation}
However, 
in the above proof for shifts, when we use Stein's inequality, we need to reduce the linearity before we can use it again.
 That is why we need
the product structure of $\UMD$ H\"older tuples rather than just a product $X_1 \times \dots \times X_n \to Y_{n+1}$ on the top level. 
\end{rem}

\section{Boundedness of multilinear paraproducts in UMD H\"older tuples}\label{sec:ParaBdd}
In this section we prove the boundedness of multilinear paraproducts.
Let us first recall a result for paraproducts acting on $\UMD$-valued functions. 
If $X$ is a $\UMD$ space, $\calD$ is a dyadic lattice and $\{a_Q\}_{Q \in \calD}$ is a collection of scalars satisfying
the $\BMO$ condition \eqref{eq:BMOCondition}, then
\begin{equation}\label{eq:LinearUMDPara}
\Big \| \sum_{Q \in \calD} a_Q \langle f \rangle_Q h_Q \Big \|_{L^p(X)}
\lesssim \| f \|_{L^p(X)},
\end{equation}
where $p \in (1, \infty)$.
This result goes back to Bourgain, see Figiel--Wojtaszczyk  \cite{FigWoj}.
Another nice proof is obtained by adapting the argument of  H\"anninen--Hyt\"onen \cite{HH}, who
consider paraproducts with operator coefficients.

Let $n \ge 1$ and let $\{X_1, \dots, X_{n+1}\}$ be a $\UMD$ H\"older tuple. 
Suppose that $\calD$ is a dyadic lattice and that $\pi{\coloneqq} \pi_\calD$ is a paraproduct as described in Section
\ref{sec:MultiSing}. Let $j_0 \in \calJ_{n+1}$ be the index related to the cancellative Haar functions of $\pi$ and let
$\si \in \Sigma(n+1)$ be the cyclic permutation such that $\si(n+1)=j_0$.
We consider the $(n+1)$-linear form $\Lambda_\pi$  acting on functions $f_j \in L^\infty_c(X_j)$
by
\begin{equation}\label{eq:ParaForm}
\Lambda_\pi(f_1, \dots, f_{n+1})
= \sum_{Q \in \calD} a_Q \tau \left(\Big[\prod_{j=1}^n \langle f_{\si(j)} \rangle_Q\Big] \langle f_{\si(n+1)}, h_Q \rangle \right),
\end{equation}
where the scalars $\{a_Q\}_{Q \in \calD}$ satisfy the $\BMO$ condition \eqref{eq:BMOCondition}.
The following theorem combined with Lemma \ref{lem:SparseModel} 
proves the desired estimate. 

\begin{thm}\label{thm:MultiPara}
Suppose that $p_j \in (1, \infty)$ for $j \in \calJ_{n+1}$ are such that $\sum_{j=1}^{n+1}1/p_j=1$.
If  $f_j \in L^\infty_c( X_{j})$ for $j \in \calJ_{n+1}$ then the form $\Lambda_\pi$ from \eqref{eq:ParaForm} satisfies the estimate
$$
|\Lambda_\pi(f_1, \dots, f_{n+1})|
\lesssim \prod_{j=1}^{n+1} \| f_j \|_{L^{p_j}(X_j)}.
$$ 
\end{thm}

\begin{proof}
For $m \in \calJ_n$ we let $p(\calJ_m)$ be the exponent defined by
$1/p(\calJ_m)= \sum_{j=1}^m 1/p_j$.
For convenience of notation we may assume that $j_0=n+1$, so that $\si=\id$.
In this case we need to estimate the term
$$
\Big \| \sum_{Q \in \calD} a_Q\prod_{j=1}^{n} \langle f_{j} \rangle_Q h_Q \Big \|_{L^{p(\calJ_n)}({Y(}X_{n+1}))}.
$$
The case $n=1$ is the known estimate \eqref{eq:LinearUMDPara}. 
Therefore, we assume that $n \ge 2$. 

Applying the $\UMD$ property of ${Y(}X_{n+1})$  we are led to 
\begin{equation}\label{eq:ParaWithSigns}
\Big(\E \int_{\R^d} \Big| \sum_{Q \in \calD} \varepsilon_Q a_Q \prod_{j=1}^{n} \langle f_j \rangle_Q |h_Q(x)| \Big |_{{Y(}X_{n+1})}^{p(\calJ_{n})} \ud x\Big)^{1/p(\calJ_{n})},
\end{equation}
where to pass from $h_Q$ to $|h_Q|$ we used that for fixed $x \in \R^d$ the families $\{ \varepsilon_Q h_Q(x)\}$ and $\{ \varepsilon_Q |h_Q(x)|\}$ 
are identically distributed. Since $|h_Q|=1_Q/|Q|^{1/2}$, we can use Stein's inequality to have that
$$
\eqref{eq:ParaWithSigns} 
\lesssim \Big(\E \int_{\R^d} \Big| \sum_{Q \in \calD} \varepsilon_Q a_Q \prod_{j=1}^{n-1} \langle f_j \rangle_Q f_{n}(x) |h_Q(x)| 
\Big |_{{Y(}X_{n+1})}^{p(\calJ_{n})} \ud x \Big)^{1/p(\calJ_{n})}.
$$

Now, we use the same inequality we used in the shift proof, Equation \eqref{eq:KeyIneq}, and H\"older's inequality  
to have that the last term is less than $\| f_n \|_{L^{p_n}(X_n)}$ multiplied by 
$$
\Big(\E \int_{\R^d} \Big| \sum_{Q \in \calD} \varepsilon_Q a_Q \prod_{j=1}^{n-1} \langle f_j \rangle_Q |h_Q(x)| 
\Big |_{Y(Y(X_1, \dots, X_{n-1}))}^{p(\calJ_{n-1})} \ud x \Big)^{1/p(\calJ_{n-1})}.
$$
Since $\{X_1, \dots, X_{n-1}, Y(X_1, \dots, X_{n-1}\}$ is a $\UMD$ H\"older $n$- tuple,
we see that we have reduced to a situation as in \eqref{eq:ParaWithSigns} but now the degree of linearity is $n-1$.
We can repeat the argument until we end up with a linear operator, and then we apply \eqref{eq:LinearUMDPara}.
\end{proof}

\section{Maximality of    UMD H\"older tuples} \label{s:maxhtup}In this brief section, we show that   $\UMD$ H\"older tuples are in a suitable sense maximal for $L^p$-boundedness of  extensions of $n$-linear CZO operators and dyadic shifts via an associative product as in  Section \ref{sec:UMDHT}. The precise statement is in Proposition \ref{p:maximal} below.

Therefore, we fix an associative algebra $\mathcal A$ and a functional $\tau$ as in  Section \ref{sec:UMDHT}. 
We begin with a lemma.

\begin{lem} \label{lem:hold} Let   $(X_1,\ldots, X_{n})$ be a $n$-tuple of admissible spaces. If  $X_{n+1}$ is an admissible space such that for all $(n+1)$-linear shift forms \eqref{e:shiftforms} and  functions $f_j\in \mathcal C^1(\R^d) \otimes X_j$, $j=1,\ldots, n+1$
\begin{equation}
\label{e:YJnorm4}
\left|\Lambda_{U^k_{\calD_\omega,u}}(f_1,\ldots, f_{n}, f_{n+1}) \right|\lesssim  \left( \prod_{\ell=1}^{n+1} \|f_j\|_{L^{n+1}(\R^d;X_j)} \right) 
\end{equation}
with implicit constant depending  possibly on the complexity $k$, then \eqref{e:YJnorm3} holds for $m=n$, and in particular  $X_{n+1}\hookrightarrow {Y(}X_1,\ldots, X_{n})$.
\end{lem}
\begin{proof} Test \eqref{e:YJnorm4} on a suitable trivial shift and appeal to Lemma \ref{lem:max}.
\end{proof}
To make our maximality claim precise, we need an additional definition. We say that the tuple $
\{X_1,\ldots,X_{n+1}\}$  of admissible spaces  \emph{is an  $n$-linear shift extension} if \eqref{e:YJnorm4} holds for all $(n+1)$-linear shift forms \eqref{e:shiftforms}.
If in addition, whenever  $Z$ is an admissible space such that for some $j_0\in \mathcal J_{n+1}$ the tuple 
$\{X_1,\ldots X_{j_0-1},Z,X_{j_0+1},\ldots X_{n+1}\}$  is an  $n$-linear shift extension, it must be $Z\hookrightarrow X_{j_0}$, we say that   $
\{X_1,\ldots,X_{n+1}\}$     is a  \emph{maximal} $n$-linear shift extension.  
\begin{prop}\label{p:maximal} Let $\{X_1, \dots, X_{n+1}\}$ be a $\UMD$ H\"older tuple. Then
\begin{itemize}
\item $\{X_1, \dots, X_{n+1}\}$ is a maximal $n$-linear shift extension;
\item whenever $1\leq k \leq n-1$ and $\#\calJ=k$,  $\{X_{j}: j \in \calJ\}\cup \{{Y(}\{X_{j}: j \in \calJ\})\}$ is   a maximal  $k$-linear shift extension. 
\end{itemize}
\end{prop}
\begin{proof}
Theorem \ref{thm:MultiShifts} shows that if  $\{X_1, \dots, X_{n+1}\}$ is a $\UMD$ H\"older tuple, then it is an  $n$-linear shift extension. As $X_{j_0}={Y(}\{X_j: j \in \mathcal J_0\})$ by definition of $\UMD$ H\"older tuple, we learn from   Lemma \ref{lem:hold} that   $\{X_1, \dots, X_{n+1}\}$ is in fact   a maximal $n$-linear shift extension. This proves the first point.

 By the inductive definition of $\UMD$ H\"older tuple, for each  $1\leq k \leq n-1$ and $\#\calJ=k$,  $\{X_{j}: j \in \calJ\}\cup \{{Y(}\{X_{j}: j \in \calJ\})\}$ is a $\UMD$ H\"older tuple.  Then this tuple must be a maximal  $k$-linear shift extension because of the first point. The second point is also proved.
 \end{proof}

\appendix
\section{Iterated mixed-norm non-commutative $L^p$ spaces} \label{a:1}
  Let $\mathcal M$ be a von Neumann algebra equipped with a n.s.f.\ trace as described in Example \ref{ex:ncex}.
Recall in particular that    $\mathcal A=L^0(\mathcal M)$ is an  associative $*$-algebra endowed with a compatible complete metrizable topology, induced by the metric $d_{\mathcal A}$ of convergence in measure. For an integer $S\geq 1$, let 
$(M_s,\mu_s)$, $s=1,\ldots, S$,  be $\sigma$-finite  measure spaces and $(\Omega_S,\omega_S)$ the product measure space
\[\Omega_S=\prod_{s=1}^S M_s,\qquad \omega_S= \prod_{s=1}^S \mu_s.\]
% In this example, as $\mathcal A$ and $\mathscr A$ below are endowed with a topology,  whenever we  write that $X$ is a Banach subspace of $\mathcal A$, we mean that the injection $X\to A$ is  continuous, and similarly for $\mathscr A$. 

 Let $\mathscr A_{0,S}$ be the vector space of \emph{simple functions} $f:\Omega_S\to \mathcal A$, namely 
\[
f (t)=\sum_{j=1}^{J } A_{j } \mathbf{1}_{E^{j}}(t), \qquad t=(t_1,\ldots, t_s)\in \Omega_S,
\]
with $A_{j}\in \mathcal A,$ $E^{j}  \subset \Omega_S$ with $\omega_S(E^j)<\infty$.  
 Then $\mathscr A_{0,S}$ is an associative algebra with respect to the pointwise product: for $f,g \in \mathscr{A}_{0,S}$, the function $fg$ defined by $(fg)(t)=f(t) g(t)$, where the latter is the strong product in $\mathcal A$, belongs to $ \mathscr A_{0,S}$.   We denote by 
\begin{equation}
\label{e:defA}
\mathscr{A}_S  {\coloneqq} \textrm{  closure of } \mathscr A_0  \textrm{ w.r.t.\    sequential } \ d_{\mathcal A}\textrm{-pointwise convergence}
\end{equation}
namely, $f \in \mathscr{A}_S $ if there exists a sequence $f_n \in \mathscr A_{0,S}$ such that
\[
\lim_{n} d_{\mathcal A} (f(t),f_n(t)) =0 \qquad a.e. \ \ t \in \Omega_S.
\]
Then $\mathscr A_S$, the class of strongly measurable $\mathcal A$-valued  functions on $\Omega_S$, is an associative  algebra with respect to the same product.  Furthermore, $\mathscr A_S$ is complete with respect to the topology of convergence in measure, namely  $f_n\to f$ if for all $\eps>0$
\[
\lim_{n} \mu\left(\left\{t\in \Omega_S: d_{\mathcal A} (f(t),f_n(t))>\eps\right\}\right) =0,
\]
and the product operation is continuous.
Note that the latter topology is also metrizable, proceeding in an analogous way to \cite[Proposition A.2.4]{HNVW1}. 

Recall that $\mathcal M$ is equipped with the n.s.f.\   trace $\tau$, which is a linear bounded functional on    $  L^1(\mathcal M)$.   Then the functional 
\[
\tau_S(f)  \coloneqq \int_{\Omega_S} \tau(f(t)) \, \mathrm{d} \omega_S(t)
\]
is linear and bounded on the      Bochner space $ L^1(\Omega_S,\omega_S; L^1(\mathcal M))$, which is a subspace of $\mathscr A_S$. With this definition, $\mathscr A_S$ is endowed with the trace $\tau_S$.
Under these assumptions, we have the following proposition.
\begin{prop} \label{p:a1} For a H\"older tuple $\{p_j^{0}:1\leq j\leq m\}$   as in \eqref{e:holder}, let
\[
X_{j}^0= L^{p_j^0}(\mathcal M).
\]
Let   $\{p_j^{s}:1\leq j\leq m\}$ be further H\"older tuples of exponents, for  $1\leq s\leq S$. Then the Banach subspaces of $\mathscr A_s$
\begin{equation}
\label{e:exumd2}
X_j^s = L^{p_j^{s}} (M_s,\mu_s;  X_j^{s-1} ), \qquad s=1,\ldots,S,
\end{equation}
are a $\UMD$ H\"older $m$-tuple.
\end{prop}
Before the proof proper, we need to set some notation, and develop suitable auxiliary lemmata. For $1\leq k \leq m-1$, $\mathcal J=\{j_1<j_2<\cdots <j_k\}\subset \mathcal J_m$, and $0\leq s\leq S$ we write
\begin{equation}
\label{e:theexp}
  \frac{1}{q_{\calJ}^s}= \sum_{u=1}^k \frac{1}{p_{j_u}^s}, \qquad 
  \frac{1}{p_{\calJ}^s}=  1 -\frac{1}{q_{\calJ}^s}.
\end{equation}
It will be convenient to introduce the auxiliary mixed norm spaces
  \[
\begin{split} &
E_{j}^1 =  L^{p_{j}^1}(M_1,\mu_1), \, 
\\
& E_{j}^s = L^{p_{j}^s} (M_s,\mu_s;E_{j}^{s-1}), \qquad s=2,\ldots, S,
\end{split}\]
for $j=1,\ldots, m$ and similarly
   \[
\begin{split} &
E_{\calJ}^0 =  \mathbb C, \, 
\\
& E_{\calJ}^s = L^{q_{\mathcal J}^s} (M_s,\mu_s;E_{\calJ}^{s-1} ), \qquad s=1,\ldots, S.
\end{split}\]
In general we write $S(X)$ for the unit sphere in the Banach space $X$.  
\begin{lem} \label{l:auxlem1} Let $\mathcal J=\{j_u:1\leq u\leq k\}$. 
There exists  maps $B_u^{s}:S(E_{\calJ}^s)\to S(E_{j_u}^s)$  
such that    
\begin{equation}
\label{e:factor1}
f = \prod_{u=1}^k B_u^s(f)  \qquad \forall f\in S(E_{\calJ}^s)
\end{equation} 
and   
\begin{equation}
\label{e:factor2}\begin{split} 
&\|f_n -f \|_{E_{\calJ}^{s }}\to 0, \; \|f_n(t_s)-f(t_s)\|_{E_{\calJ}^{s-1}} \to 0 \; \mathrm{a.e.} \; t_s\in M_s  \implies
\\ &
\|B_u^s(f)  -B_u^s(f_n)  \|_{E_{j_u}^{s }}\to 0, \; \|B_u^s(f_n) (t_s)-B_u^s(f_n) (t_s)\|_{E_{j_u}^{s-1}} \to 0 \; \mathrm{a.e.} \; t_s\in M_s, \quad 1\leq u\leq k.
\end{split}
%\begin{rcases}
% & \|f_n(t_s)-f(t_s)\|_{E_{\calJ}^{s-1}} \to 0  
%\\
%&
%\end{rcases}
\end{equation}
\end{lem}
\begin{proof} We deal with the case $j_u=u,u=1,\ldots, k$ which is generic.  We prove the statement by induction on $s$. 
 If  $s\geq 2$, assume inductively that  maps  $B_u^{s-1}$ as in the statement have been constructed; for the base case $s=1$, we run the argument below with $B^{0}_u$ the identity map.  In both cases, we need to define    $B^s_u:S(E_{\calJ}^s)\to S(E_{u}^s)$. We use that each $f\in S(E_{\calJ}^s) $ is $E_{\calJ}^{s-1}$-valued. So for each  $t_s\in M_s$, write 
\[
f (t_s ) = |f(t_s )|_{ E_{\calJ}^{s-1}} g(t_s) = \prod_{u=1}^k \left( |f (t_s ) |_{E_{\calJ}^{s-1}}^{\frac{q_{\calJ}^s }{p_u^s}} g_{u}(t_s) \right)\eqqcolon \prod_{u=1}^k B^s_u(f)(t_s )
\]
where $g$ is $S( E_{\mathcal J}^{s-1})$-valued, so that each $g_u = B_u^{s-1}(g)$ is $S(E_{u}^{s-1})$-valued. Notice that each $f_{u}=B^s_u(f)$ is  (strongly) $\mu_s$-measurable with values in $E_{u}^{s-1 }$: in fact $|f (\cdot ) |_{E_{\calJ}^{s-1}}$ is $\mu_s$-measurable   and each $g_u$ is   $\mu_s$-measurable, as  $B_u^{s-1}$ is (norm) continuous from $E_\calJ^{s-1} \to E_{u}^{s-1 }$ and $g$ is $\mu_s$-measurable with values in $E_\calJ^{s-1}$ . A direct calculation reveals that
\[
\|f_{u}\|_{E_{u}^s}=1, \qquad 1\leq u\leq k.
\]
%so that we have achieved
%\[
%f= \prod_{u=1}^k   B_u^s(f), \qquad \|f_{u}\|_{E_{u}^s}=1.
%\]
It remains to show that the thus defined maps  $  B_u^s $ are continuous in the sense of \eqref{e:factor2} by assuming the same properties hold  for the maps $  B_u^{s-1} $. Let $f_n, f$ be as in the first line of \eqref{e:factor2} and 
  write $f_n(t_s)=|f_n(t_s)|_{E_{\calJ}^{s-1}} g_n(t_s)$. We first show the pointwise convergence:  for each we have
\[
\begin{split}
  \|{B}^s_u(f)(t_s)-B^s_u(f_n)(t_s)\|_{E_{u}^{s-1}} & \leq  \|f(t_s)\|_{E_{\calJ}^{s-1}}\|{B}^{s-1}_u(g(t_s))-B^{s-1}_u(g_n(t_s))\|_{E_{u}^{s-1}}\\ & + \quad
\|B^{s-1}_u(g_n(t_s))\|_{E_{u}^{s-1}}   \big|\|f(t_s)\|_{E_{\calJ}^{s-1}}^{\frac{q_{\calJ}^s}{p_u^s}} - \|f_n(t_s)\|_{E_{\calJ}^{s-1}}^{\frac{q_{\calJ}^s}{p_u^s}}  \big|
 \end{split}\]
  Relying on the norm continuity of ${B}^{s-1}_u$  we obtain that both summands in the previous display converge to zero for each $t_s$
such that $\|f_n(t_s)\|_{E_{\calJ}^{s-1}} \to \|f(t_s)\|_{E_{\calJ}^{s-1}}, \|g_n(t_s)- g(t_s)\|_{E_{\calJ}^{s-1}}\to 0$; this is a set of full $\mu_s$ measure, so that this part of the proof is complete. We come to the norm continuity in \eqref{e:factor2}. We have 
\[
\begin{split} 
 \| {B}^s_u(f)-B^s_u(f_n)\|_{E_{u}^s}^{p_u^s}&\lesssim \int_{M_s} |f(t_s) |_{E_{\calJ}^{s-1}}^{q_{\calJ}^s} |B_u^{s-1}(g(t_s))- B_u^{s-1}(g_n(t_s))|^{p_u^s}_{E_{u}^{s-1}}\, \mathrm{d} \mu_s(t_s) \\ &\quad  +  \int_{M_s} \left| |f(t_s)|_{E_{\calJ}^{s-1}}^{\frac{q_{\calJ}^s}{p_u^s}}-|f_n(t_s)|_{E_{\calJ}^{s-1}}^{\frac{q_{\calJ}^s}{p_u^s}} \right|^{p_u^s}|B_u^{s-1}(g_n(t_s))|^{p_u^s}_{E_{u}^{s-1}}\, \mathrm{d} \mu_s(t_s)
\end{split}\]
The first integrand converges to zero pointwise a.e. and is dominated by $|f(t_s) |_{E_{\calJ}^{s-1}}^{q_{\calJ}^s}$, so the integral converges to zero by dominated convergence. The second integral is equal to\[
 \|F -F_n \|^{p_u^s}_{L^{p_u^s}(M_s,\mu_s)}, \qquad F(t_s)=|f (t_s)|_{E_{\calJ}^{s-1}}^{\frac{q_{\calJ}^s}{p_u^s}}, \qquad F_n(t_s)=|f_n(t_s)|_{E_{\calJ}^{s-1}}^{\frac{q_{\calJ}^s}{p_u^s}}.
\]
Notice that $\|F \|_{p_u^s}= \|f\|_{E_{\calJ}^{s}}^{{q_\calJ^s}/{p_u^s}}$, $\|F_n \|_{p_u}= \|f_n\|_{E_{\calJ}^{s}}^{{q_\calJ^s}/{p_u^s}}$.
{ As $F_n\to F$ pointwise}, $F_n,F\in L^{p_u^s}(M_s,\mu_s) $ and $ \|F_n \|_{p_u^s}\to  \|F  \|_{p_u^s}$, then $\|F -F_n \|_{p_u^s}$ converges to zero by a well-known variation of the proof of the $L^p$ dominated convergence theorem.
%\footnote{This is an easy argument so I think we should omit it. For the record, the argument is the following: let
%$
%h_n= 2^pF^p +2^pF_n^p -|F_n-F|^p 
%$. This is a sequence of nonnegative measurable functions whose pointwise limit is $2^{p+1} F^p$. Then by Fatou
%\[
%2^{p+1} \int F^p = \int \lim h_n \leq \liminf \int h_n =  2^{p+1} \int F^p -\limsup \int |F_n-F|^p 
%\]
%and as the left hand side is finite we may subtract it from the last right hand side and get the conclusion.
%}.
\end{proof}

\begin{lem} \label{l:auxlem2} Let
\footnote{Recall that $L^{q_{\mathcal J}^0}(\mathcal M)_+$ denotes the positive cone of $L^{q_{\mathcal J}^0}(\mathcal M),$ namely the positive operators in $L^{q_{\mathcal J}^0}(\mathcal M)$.}
\[
\begin{split} &
X_{\calJ}^0 =  L^{q_{\mathcal J}^0}(\mathcal M),  \qquad X_{\calJ,+}^0= L^{q_{\mathcal J}^0}(\mathcal M)_+,
\\
& X_{\calJ}^s = L^{q_{\mathcal J}^s} (M_s,\mu_s;X_{\calJ}^{s-1}),\qquad X_{\calJ,+}^s = L^{q_{\mathcal J}^s} (M_s,\mu_s;X_{\calJ,+}^{s-1}), \qquad s=1,\ldots, S.
\end{split}\]
Let $f\in X_{\calJ,+}^s  $ be a simple function with  $\|f\|_{X_{\calJ}^s}=1$. Then there exist $f_u \in X_{j_u,+}^s$,
 $u=1,\ldots, k$ with
\begin{equation}
\label{e:dec-1}
f=\prod_{u=1}^k f_u,  
\qquad
\|f_u\|_{X_{j_u}^S}=1.
\end{equation}
\end{lem}
\begin{proof}Again we deal with the generic  case $j_u=u,u=1,\ldots, k$. First of all, we make a remark about the case $s=0$. Fix $A\in X_{\calJ,+}^0$ with $\|A\|_{X_\calJ^0}=1$. 
%As in particular $A\in \mathcal A$, which is a $*$-subalgebra of the closed, densely defined operators on $H$ affiliated with $\mathcal M$,   there exists $U\in \mathcal M$ such that $A=U(A^*A)^{\frac 12}=:U|A|$ and $U$ satisfies that $U^*U=P_{({\rm{ker}} A)^\perp}$: a  precise statement may be found in \cite[Theorem 2.3.2, Proposition 2.3.3]{VNA}.
Using the Borel functional calculus for  positive closed densely defined operators to define $A^\theta $ for $\theta>0$
\begin{equation}
\label{e:dec0}
A= \prod_{u=1}^k B_u(A), \qquad    B_u(A)= A^{\frac{q_\calJ^0}{p_{u}^0}}, \quad u=1,\ldots,k.
\end{equation}
%and we immediately see that $|A_1|=(|A|^{\frac{q_\calJ}{p_{j_1}}} U^* U|A|^{\frac{q_\calJ}{p_{j_1}}})^{\frac 12}=|A|^{\frac{q_\calJ}{p_{j_1}}}$ and 
Trivially   
\[
 \|B_u(A)\|_{X_u^0}=\|A\|^{\frac{q_\calJ^0}{p_{u}^0}}_{X_\calJ^0}=1, \quad u=1,\ldots,k.
\]
We now prove the main  statement.  Let $f\in X_{\calJ,+}^s  $ be a simple function with  $\|f\|_{X_{\calJ}^s}=1$. We factor
\[
f(t) = F(t) A(t), \qquad F(t)= |f(t)|_{X_\calJ^0}, \qquad t \in \Omega_s.
\]
Notice that $F\in E_{\mathcal J}^s$ of unit norm, so that using Lemma \ref{l:auxlem1}
\[
F= \prod_{u=1}^k B_u^s(F) , \qquad  \|B_u^s(F)\|_{E_{u}^s}=1,
\]
and we may write, also using \eqref{e:dec0}
\[
f= \prod_{u=1}^k f_u, \qquad f_u(t)=B_u^s(F)(t) B_u(A(t)),
\]
Notice that each $f_u$ is  strongly measurable as $B_u(A(\cdot))$ is a simple $X_{u,+}^0$-valued function and $B_u^s(F)$ is a measurable function in $E_u^s$. Also as $|B_u(A(t))|_{X_u^0}=1$ for all $t\in \Omega_s$
\[
\|f_u\|_{X_u^s} = \|B_u^s(F)\|_{E_{u}^s}=1,
\]
which completes the proof of the claim.
\end{proof}
We turn to the proof of the proposition. Namely we need to show that the tuple $X_j^s$ from \eqref{e:exumd2} is a $\UMD$ H\"older tuple for each $s=1,\ldots, S$.
 In proving this, by virtue of the case $s=0$ being already established in Example \ref{ex:ncex} we may argue inductively and assume the claim has been proved in the cases of  $0,\ldots, s-1$.

 Clearly each $  X_j^s$ is a subspace of $ \mathscr A_s$.  
Denoting by $q_j^s, s=0,\ldots, S$ the conjugate exponent  of $p_j^s$, it is convenient to define the spaces
\[\begin{split}
&Y_{j}^0 =  L^{q_{j}^0}(\mathcal M), \, 
\\
& Y_{j}^s = L^{q_{j}^{s}} (M_s,\mu_s;Y_{j}^{s-1} ), \qquad s=1,\ldots, S.
\end{split}
\]
which are Banach subspaces of $ \mathscr A_s$. 
Further, as  each $X_j^s$  is a reflexive Banach space and enjoys the Radon-Nikod\'ym property \cite[Theorem 1.3.21]{HNVW1}, an inductive argument yields the Riesz representation theorem  (cf.  \cite[Theorem 1.3.10]{HNVW1}) then yields that
\begin{equation}
\label{e:factabove}
\left(X_j^s\right)^*=Y_j^s, \qquad 1\leq j\leq m
\end{equation}
through the identification
\[
\lambda \in (X_j^s)^* \mapsto g_\lambda  \in Y_j^s  \qquad 
\lambda(f) = \tau_s( g_\lambda f) , \quad f   \in  X_j^s.
\]
We have in particular shown that each $X_j^s$ is an admissible space for the algebra $\mathscr A_s$ with trace $\tau_s$ and $Y(X_j^s)=Y_j^s$ .

We verify that $\{ X_j^s:j \in \mathcal J_m\}$ is a $\UMD$ H\"older tuple by induction on $m$. The case $m=2$ is actually immediate by virtue of the observation  and the well known fact that each $X_j^s,Y_j^s$ is a $\UMD$ space.

To obtain the inductive step, we fix $m\geq 3$ and  verify the following equality.  For each $1\leq k \leq m-1$, $\mathcal J=\{j_1<j_2<\cdots <j_k\}\subset \mathcal J_m$, there holds
\begin{equation}
\label{e:equality} Y( \{ X_{j}^s:j\in \mathcal J\}) \textrm{ is isometrically isomorphic to }
  \left( X_{\calJ}^s\right)^* ,
\end{equation}
where we refer to the spaces defined in Lemma \ref{l:auxlem2}. More explicitly, denoting
\[
\begin{split} &
Y_{\calJ}^0 =  L^{p_{\mathcal J}^0}(\mathcal M),
\\
& Y_{\calJ}^s = L^{p_{\mathcal J}^s} (M_s,\mu_s;Y_{\calJ}^{s-1}),  \qquad s=1,\ldots, S,
\end{split}\] 
we have $Y(\{ X_{j}^s:j\in \mathcal J\})=Y_{\calJ}^s= \left( X_{\calJ}^s\right)^*   $.

 Property  P1 then corresponds to this equality in the cases $k= m-1$. Verifying property P2   amounts to checking that when $k<m-1$, the tuple   $\{  X^s_j:j\in \calJ\} \cup \{   Y^s_{\calJ}\} $ is a  $\UMD$ H\"older $(k+1)$-tuple. As $k< m-1$, $\{  X_j^s:j\in \calJ\} \cup \{   Y^s_{\calJ}\} $ is a $\UMD$ H\"older $(k+1)$-tuple   and the exponents
$\{p_j^s: j \in \mathcal J, p^s({\calJ})\}$ are a H\"older tuple, this check is made by a straightforward  appeal to the induction assumption.

We are left with proving \eqref{e:equality}. To do this we will define a linear surjective isometry $\Phi:Y( \{ X_{j}^s:j\in \mathcal J\})\to Y_{\mathcal J}^{s } $. First of all note that 
\begin{equation}
\label{e:n0}
\|g\|_{ Y( \{ X_{j}^s:j\in \mathcal J\})} \leq 
\|g\|_{L^{p_\calJ^s}(M_s,\mu_s; Y_{\mathcal J}^{s-1})} = \|g\|_{  Y_{\mathcal J}^{s }}
\end{equation}
descends immediately  from H\"older's inequality in $L^{p}(M_s,\mu_s)$-spaces and Lemma \ref{lem:max} applied to the $\UMD$ H\"older tuple $X_{j_1}^{s-1}, X_{j_2}^{s-1},\ldots,  X_{j_k}^{s-1}$. 
We will use this below.

Fix then  $g\in Y( \{ X_{j}^s:j\in \mathcal J\}) $. We claim that  if $f$ is a simple $X_{\calJ,+}^0$-valued function on $\Omega_s$ with  $\|f\|_{X_{\calJ}^s}=1$, then 
\begin{equation}
\label{e:n01}
|\tau_s(g f)|\leq \|g\|_{ Y( \{ X_{j}^s:j\in \mathcal J\})} .
\end{equation}
Indeed, applying Lemma \ref{lem:max} we obtain 
\[
\left|\tau_s(g f)\right| = \left|\tau_s\left(g \prod_{u=1}^k f_u\right)\right| \leq \|g\|_{ Y( \{ X_{j}^s:j\in \mathcal J\})} \prod_{u=1}^k \|f_u\|_{X_{j_u}^s},  
\qquad
\|f_u\|_{X_{j_u}^s} = 1, \quad u=1,\ldots, k,
\]
which is \eqref{e:n01}.
As $X_\calJ^s$ is the $ X_\calJ^s$-norm closure of the linear span of   simple $X_{\calJ,+}^0$-valued function on $\Omega_s$, the linear bounded functional  $f\mapsto\tau_s(g f)$ extends uniquely to an element $\Phi(g) $ of $(X_\calJ^s)^*\equiv Y_\calJ^s$ with \[
 \|\Phi(g)\|_{ Y_\calJ^s}
 \leq  \|g\|_{ Y( \{ X_{j}^s:j\in \mathcal J\})}.
 \]
 It is easy to see that the map $\Phi:Y( \{ X_{j}^s:j\in \mathcal J\})\to Y_{\mathcal J}^{s } $ is linear. From \eqref{e:n0} we gather that if $g \in Y_{\mathcal J}^{s }$ then  $\Phi(g)$ is well-defined. In this case the linear bounded functionals $g\mapsto\tau_s(g f)$  and $\Phi(g)$ coincide on a dense set, it must be $\Phi(g)=g$. So $\Phi$ is obviously surjective. Furthermore using  \eqref{e:n0} again we obtain 
\[
 \|\Phi(g)\|_{ Y_\calJ^s} \geq  \|\Phi(g)\|_{ Y( \{ X_{j}^s:j\in \mathcal J\})} = \|g\|_{ Y( \{ X_{j}^s:j\in \mathcal J\})} \geq 
 \|\Phi(g)\|_{ Y_\calJ^s}
\] whence equality must hold throughout. So $\Phi$ is a linear isometric isomorphism and the proof of \eqref{e:equality} is complete.
\bibliography{OP_Multilin}
\bibliographystyle{amsplain}

\end{document}